\documentclass[12pt,twoside]{article}

\usepackage{amsmath}
\allowdisplaybreaks[4]
\usepackage{amsfonts}
\usepackage{dsfont}
\usepackage{amssymb}
\usepackage{mathrsfs}
\usepackage{amsfonts}
\usepackage{graphicx}
\usepackage{epstopdf}
\usepackage{subfigure}
\usepackage{xcolor}
\usepackage{epstopdf}
\usepackage{epsfig}
\usepackage{enumerate}
\usepackage{amsthm}
\usepackage{cases}
\usepackage{hyperref}
\usepackage{caption}
\usepackage[noend]{algpseudocode}
\usepackage{algorithmicx,algorithm}
\usepackage[titletoc]{appendix}
\UseRawInputEncoding
\numberwithin{equation}{section}

\makeatletter 
\newtheorem{theorem}{Theorem}[section]
\newtheorem{lemma}{Lemma}[section]
\newtheorem{example}{Example}[section]
\newtheorem{remark}{Remark}[section]

\allowdisplaybreaks

\allowdisplaybreaks

\begin{document}

\date{}
\author{Caiyu Jiao,\;Changpin Li\thanks{\,Corresponding author.}
\\
\small \textit{ Department of Mathematics, Shanghai University, Shanghai 200444, China
}
}
\title{Monte Carlo method for parabolic equations involving fractional Laplacian\thanks{The work was partially supported by the
National Natural Science Foundation of China under Grant no. 11926319.}}
\maketitle
\hrulefill

\begin{abstract}
 {We apply the Monte Carlo method to solving the Dirichlet problem of linear parabolic equations with fractional
 Laplacian. This method exploits the idea of weak approximation  of related stochastic differential equations
 driven by the symmetric stable L\'{e}vy  process with jumps. We utilize the jump-adapted scheme to approximate
 L\'{e}vy process  which gives exact exit time to the boundary. When the solution has low regularity, we establish
 a numerical  scheme by removing the small jumps of the L\'{e}vy process and then show the convergence order. When
 the solution has higher regularity, we build up a higher-order numerical scheme by replacing small jumps with a
 simple process and then display the higher convergence order. Finally, numerical experiments including
 ten- and one hundred-dimensional cases are presented, which confirm the theoretical estimates and show the numerical
 efficiency of the proposed schemes for high dimensional parabolic equations.
 }
\end{abstract}
\hrulefill

\section{Introduction}

The fractional Laplacian, $(-\Delta)^{s}$, is a prototypical operator for modeling
nonlocal and anomalous phenomenon which incorporates long range interactions
\cite{Antil&Harlim2021,Guo2013,Hu2022,Li&Cai2019,Metzler&Klafter2000,Sheng&Wang2021,Zhai&Zheng2019}.
It arises in many areas of applications,
including models for turbulent flows, porous media flows,
pollutant transport, quantum mechanics, stochastic dynamics, finance and so on
\cite{Das2011,Herrmann2011,Hilfer2000,Jin&Yan2019,Oldham&Spanier1974}.
Almost known deterministic numerical methods have been proposed for approximating
solutions of parabolic problems with fractional Laplacian in low dimension
(less than 3 dimension)\cite{Aceto&Novati2017,Alzahrani&Khaliq2019},
while seldom probabilistic approach is taken
into account to numerically solve such parabolic and steady state problems.

  The aim of this article is to develop a Monte Carlo method for solving
the terminal-boundary value problem in high dimensional cases in the following form
\begin{equation}\label{eq:terminalboundaryproblem}
 \left\{
 \begin{aligned}
  &\frac{\partial u}{\partial t}-(-\Delta)^su+b(t,\texttt{{\rm{x}}})
  \cdot{\nabla u}+c(t,\texttt{{\rm{x}}})u+
  f(t,\texttt{{\rm{x}}})=0, \quad (t,\texttt{{\rm{x}}})\in[0,T)\times \mathbb{D},
  \\[3pt]
  &u(T,\texttt{{\rm{x}}})=g(\texttt{{\rm{x}}}), \quad \texttt{{\rm{x}}}\in {\mathbb{D}},
  \\[3pt]
  &u(t,\texttt{{\rm{x}}})=\chi(t,\texttt{{\rm{x}}}), \quad (t,\texttt{{\rm{x}}})\in[0,T]
  \times (\mathbb{R}^{n}\setminus\mathbb{D}),
 \end{aligned}
 \right.
\end{equation}
where $s\in(0,1)$, $T>0$, $b(t,\texttt{{\rm{x}}})=(b_{1}(t,\texttt{{\rm{x}}}),\,
b_{2}(t,\texttt{{\rm{x}}}),\,\cdots,\,b_{n}(t,\texttt{{\rm{x}}}))$ is an $n$-dimensional vector,
$\mathbb{D}$ is a bounded region in $\mathbb{R}^{n}\,(n\geq3)$
and the fractional Laplacian is defined by a singular integral which coincides
with Riesz derivative on the whole space \cite{Cai&Li2019,Cai&Li2019a,Jiao2021,Li-Li2022,Sabelfeld2008,Samko&Kilbas1993},
\begin{equation}\label{eq:1pvfractionLaplacian}
 (-\Delta)^{s}u(t,\texttt{{\rm{x}}})=C(n,s)\,
 {\rm P.V.}\int_{\mathbb{R}^{n}}\frac{u(t,\texttt{{\rm{x}}})
 -u(t,\texttt{{\rm{y}}})}{|\texttt{{\rm{x}}}
 -\texttt{{\rm{y}}}|^{n+2s}}
 {\rm d}\texttt{{\rm{y}}}.
 \end{equation}
 Here ${\rm P.V.}$ stands for the principle value and the constant
 $C(n,s)$ is given by \cite{Cai&Li2019},
 \begin{equation}\label{eq:coeffient}
 C(n,s)=\left(\,\,\int_{\mathbb{R}^{n}}\frac{1-\cos\zeta_{1}}
 {|\zeta|^{n+2s}}{\rm d}\zeta\right)^{-1}
 =\frac{{s}2^{2s}\Gamma(\frac{n}{2}+s)}{{\pi}^{\frac{n}{2}}\Gamma(1-s)}
 \end{equation}
 with $\zeta_{1}$ being the first component of
 $\zeta=(\zeta_{1},\zeta_{2},\ldots,\zeta_{n})\in\mathbb{R}^{n}$
 and $\Gamma$ representing the Gamma function.

If we let $v(t,\texttt{{\rm{x}}})=u(T-t,\texttt{{\rm{x}}})$,
then equation (\ref{eq:terminalboundaryproblem})
is changed to the initial-boundary value problem
\begin{equation}\label{eq:initalboundaryproblem}
 \left\{
 \begin{aligned}
  &\frac{\partial v}{\partial t}+(-\Delta)^sv+\overline{b}(t,\texttt{{\rm{x}}})\cdot
  \nabla v+\overline{c}(t,\texttt{{\rm{x}}})v+
  \overline{f}(t,\texttt{{\rm{x}}})=0, \quad (t,\texttt{{\rm{x}}})\in(0,T]\times \mathbb{D},
  \\[3pt]
  &v(0,\texttt{{\rm{x}}})=g(\texttt{{\rm{x}}}), \quad \texttt{{\rm{x}}}\in {\mathbb{D}},
  \\[3pt]
  &v(t,\texttt{{\rm{x}}})=\overline{\chi}(t,\texttt{{\rm{x}}}), \quad (t,\texttt{{\rm{x}}})\in[0,T]
  \times (\mathbb{R}^{n}\setminus\mathbb{D}),
 \end{aligned}
 \right.
\end{equation}
where $\overline{b}(t,\texttt{{\rm{x}}})=-b(T-t,\texttt{{\rm{x}}})$,
$\overline{c}(t,\texttt{{\rm{x}}})=-c(T-t,\texttt{{\rm{x}}})$,
$\overline{f}(t,\texttt{{\rm{x}}})=-f(T-t,\texttt{{\rm{x}}})$, and
$\overline{\chi}(t,\texttt{{\rm{x}}})=\chi(T-t,\texttt{{\rm{x}}})$.

Probabilistic numerical methods (usually implemented with Monte Carlo method) provide
effective approaches for numerically solving partial differential equation in
high dimension with/without fractional Laplacian as probabilistic methods
do not require any discretization in space \cite{Jiao2022,Sabelfeld1991,Sabelfeld2008,Sabelfeld2017,Sabelfeld2019,Sabelfeld-Shalimova2013,Sabelfeld-Simonov2016}.

 Classical Laplace operator and fractional Laplacian are the infinitesimal
generators of Brownian motion and symmetric $2s$-stable process, respectively,
which connect partial differential equations with stochastic processes.
Muller \cite{Muller1956} first proposed walk-on-spheres method by simulating
the paths of the Brownian motion in spheres to numerically solve the steady-state
equation with classical Laplace operator. Kyprianou et al.
\cite{Kyprianou&Osojnik2018} utilized walk-on-spheres
method to approximate solution of the steady-state equation
with fractional Laplacian based on the distribution of symmetric $2s$-stable process
issued from the origin, when it first exits a unit sphere. However, the walk-on-spheres
method generally can not generate the first exit time from the domain
$\mathbb{D}$. Therefore, it is difficult for walk-on-spheres method to solve parabolic
problems numerically. Milstein and Tretyakov \cite{Milstein&Tretyakov2001,Milstein&Tretyakov2007}
approximated the trajectory of diffusion process with Brownian motion by
Euler scheme which is a uniformly-time discretization
scheme to solve classical parabolic problems with integer order derivative.
\cite{Deligiannidis2021} gave a random walk algorithm for the Dirichlet
problem for parabolic integro-differential equation where the kernel of
the integro-differential operator has better regularities
than the kernel of fractional Laplacian.

In this article, we can approximate the trajectory of the system of stochastic
differential equations with symmetric $2s$-stable process to numerically
solve fractional parabolic problems (\ref{eq:terminalboundaryproblem}).
However, simulating this system via Euler scheme suffers from two
difficulties: First, the jump intensity of the symmetric $2s$-stable process
is infinite, which means there is an infinite number of jumps in every
interval of nonzero length; Second, the exit time of the process leaving
the domain $\mathbb{D}$ is hard to obtain. To overcome the first difficulty,
we utilize the idea of Asmussen and Rosi\'{n}ski \cite{Asmussen&Rosinski2001}
and remove small jumps or replace small jumps with corresponding simple processes.
Jump-adapted scheme \cite{Deligiannidis2021,Higa&Tankov2010,Mikulevicius2012} is
adapted to go through the second difficulty.
Compared with classical Euler scheme, jump-adapted scheme uses adaptive
non-uniform discretization based on the times of jumps of the driving process.

In this paper, we first give the probabilistic representation of the solution to equation
(\ref{eq:terminalboundaryproblem}), which is deeply connected with the system
of stochastic differential equations with symmetric $2s$-stable process. We then consider
a simple jump-adapted Euler scheme to approximate the trajectory of the stochastic
process and obtain the numerical solution to equation (\ref{eq:terminalboundaryproblem}).
Furthermore, we propose a high-order jump-adapted scheme by replacing small jumps with
simple process if the regularity of the solution $u(t,\texttt{{\rm{x}}})$ is good enough.
In addition, we give the weak convergence of the simulated L\'{e}vy process and
the error estimate , which is related to the jumping intensity
and statistical error. In comparison with \cite{Deligiannidis2021},
we study the Dirichlet problem for the parabolic problem with fractional
Laplacian where the kernel is more singular and the related symmetric 2s-stable
process do not exist any moments. Based on the different regularities of the solution
$u(t,\texttt{{\rm{x}}})$, we give the corresponding numerical schemes.
For the optimal error estimate of the numerical scheme, we require
$u\in C^{1,3}([0,T]\times \mathbb{R}^{n})$, while \cite{Deligiannidis2021}
requires a solution of the auxiliary Dirichlet problem belonging to
$C^{2,4}([0,T]\times \mathbb{R}^{n})$.

The rest of the paper is organized as follows. In Section 2, the probabilistic
representation of the solution to equation (\ref{eq:terminalboundaryproblem})
is given, which is related to the system of stochastic differential equations
with $2s$-stable process. In Section 3, A simple jump-adapted Euler
scheme is derived. A high-order jump-adapted scheme is proposed in Section 4.
The weak convergence of the simulated process and error estimates
are presented in the corresponding sections. In Section 5, numerical
experiments are performed to verify the theoretical analysis. Finally,
we summarize the main work in the last section.

In the following sections, we denote positive constants by $C_{1}$ and $C_{2}$
which may be dependent of the index $s$, but not necessarily
the same at different situations.

\section{Probabilistic representation}
Let $(\Omega,\mathcal{F},\mathrm{P})$ be a complete probability space with a
filtration $F=\{\mathcal{F}_{t}\}_{t\in[0,T]}$ of $\sigma$-algebra satisfying
the usual conditions. Consider the symmetric $2s$-stable process \cite{Applebaum2009},
\begin{equation}\label{levyprocess}
 {\rm d}L_{\eta}=\int_{|\texttt{{\rm{y}}}|<1}\texttt{{\rm{y}}}
  \widetilde{N}({\rm d}\eta,{\rm d}\texttt{{\rm{y}}})+\int_{|\texttt{{\rm{y}}}|\geq1}
  \texttt{{\rm{y}}} N({\rm d}\eta,{\rm d}\texttt{{\rm{y}}}),
\end{equation}
where $N({\rm d}\eta,{\rm d}\texttt{{\rm{y}}})$ is a Poisson random measure
on $[0,T]\times \mathbb{R}_{0}^{n}$, ($\mathbb{R}_{0}^{n}=\mathbb{R}^{n}\setminus\{0\}$)
with
\begin{equation}{\label{eq:PossionRandomMeasure}}
 \mathbf{E}[N({\rm d}\eta,{\rm d}\texttt{{\rm{y}}})]=\nu({\rm d}\texttt{{\rm{y}}})
 {\rm d}\eta=C(n,s)\frac{{\rm d}\texttt{{\rm{y}}}}{|\texttt{{\rm{y}}}|^{n+2s}}{\rm d}\eta
\end{equation}
and
\begin{equation}
 \widetilde{N}({\rm d}\eta,{\rm d}\texttt{{\rm{y}}})=N({\rm d}\eta,{\rm d}\texttt{{\rm{y}}})
 -\nu({\rm d}\texttt{{\rm{y}}}){\rm d}\eta
\end{equation}
is the compensated Poisson random measure. The characteristic function is given
by \cite{Sato1999},
\begin{equation}
 \begin{aligned}
  \mathbf{E}e^{i(\xi,L_{t})}
  &=\exp\left[t\int_{\mathbb{R}^{n}}\left(e^{i(\xi,\texttt{{\rm{y}}})}
  -1-i(\xi,\texttt{{\rm{y}}})\right)\nu({\rm d}\texttt{{\rm{y}}})\right]
  \\
  &=\exp\left[tC(n,s)\int_{\mathbb{R}^{n}}\frac{\cos{(\xi,\texttt{{\rm{y}}})}-1}
  {|\texttt{{\rm{y}}}|^{n+2s}}{\rm d}\texttt{{\rm{y}}}\right]
  \\
  &=e^{-t|\xi|^{2s}}.
 \end{aligned}
\end{equation}
It can be easily got that the infinitesimal generator of $L_{t}$
is $-(-\Delta)^{s}$ \cite{Applebaum2009} . Thus, fractional Laplacian
is closely related to the symmetric $2s$-stable process.

Consider the following system of stochastic differential equations,
\begin{equation}\label{SDEs}
 \left\{
 \begin{aligned}
  &{\rm d} \texttt{{\rm{X}}}_{\eta}=b(\eta,\texttt{{\rm{X}}}_{\eta-}){\rm d}\eta
  +{\rm d}L_{\eta},
  \qquad &&\texttt{{\rm{X}}}_{t}=\texttt{{\rm{x}}},
  \\[3pt]
  &{\rm d}Y_{\eta}=c(\eta,\texttt{{\rm{X}}}_{\eta-})Y_{\eta}{\rm d}\eta, \qquad &&Y_{t}=1,
  \\[3pt]
  &{\rm d}Z_{\eta}=f(\eta,\texttt{{\rm{X}}}_{\eta-})Y_{\eta}{\rm d}\eta,  \qquad &&Z_{t}=0.
 \end{aligned}
 \right.
\end{equation}
Based on the above system, we give the probabilistic representation
of the solution to equation (\ref{eq:terminalboundaryproblem}) in the following theorem.

\begin{theorem}
 Let $u(t,\texttt{{\rm{x}}})$ be the solution of equation
 {\rm(\ref{eq:terminalboundaryproblem})}. Then $u(t,\texttt{{\rm{x}}})$ can be given by
 \begin{equation}
  \begin{aligned}
   u(t,\texttt{{\rm{x}}})
   =&\,\mathbf{E}\left[
   u(T\wedge\tau_{t,\texttt{{\rm{x}}}},\texttt{{\rm{X}}}_{T\wedge\tau_{t,\texttt{{\rm{x}}}}}
   ^{t,\texttt{{\rm{x}}}})  Y_{T\wedge\tau_{t,\texttt{{\rm{x}}}}}
   ^{t,\texttt{{\rm{x}}},1}+Z_{T\wedge\tau_{t,\texttt{{\rm{x}}}}}
   ^{t,\texttt{{\rm{x}}},1,0}\right]
   \\[3pt]
   =&\,\mathbf{E}\left\{I_{\{\tau_{t,\texttt{{\rm{x}}}}<T\}}
   \left[\chi(\tau_{t,\texttt{{\rm{x}}}},\texttt{{\rm{X}}}_{\tau_{t,\texttt{{\rm{x}}}}}
   ^{t,\texttt{{\rm{x}}}})  Y_{\tau_{t,\texttt{{\rm{x}}}}}
   ^{t,\texttt{{\rm{x}}},1}+Z_{\tau_{t,\texttt{{\rm{x}}}}}
   ^{t,\texttt{{\rm{x}}},1,0}\right]\right\}
   \\[3pt]
   &+\mathbf{E}\Big\{I_{\{\tau_{t,\texttt{{\rm{x}}}}\geq T\}}\Big[
   g(\texttt{{\rm{X}}}_{T}^{t,\texttt{{\rm{x}}}})Y_{T}^{t,\texttt{{\rm{x}}},1}+Z_{T}
   ^{t,\texttt{{\rm{x}}},1,0}\Big]\Big\},
  \end{aligned}
 \end{equation}
 where $\texttt{{\rm{X}}}_{\eta}^{t,\texttt{{\rm{x}}}}$, $Y_{\eta}^{t,\texttt{{\rm{x}}},y}$
and $Z_{\eta}^{t,\texttt{{\rm{x}}},y,z}$ $(t \leq \eta\leq T)$
are the solution of Cauchy problem {\rm(\ref{SDEs})},
$\tau_{t,\texttt{{\rm{x}}}}=\{\eta\geq t, \texttt{{\rm{X}}}_{\eta}^{t,\texttt{{\rm{x}}}}\notin
\mathbb{D} \}$ is the first exit time of $\texttt{{\rm{X}}}_{\eta}^{t,\texttt{{\rm{x}}}}$
starting from $\texttt{{\rm{x}}}\in{\mathbb{R}^{n}}$ to the boundary $\mathbb{D}$.
\end{theorem}

\begin{proof}
 From It\^{o} formula, we have
 \begin{equation}
  \begin{aligned}
   {\rm d}u(\eta,\texttt{{\rm{X}}}_{\eta})
   =&\,\frac{\partial{u}}{\partial \eta}(\eta,\texttt{{\rm{X}}}_{\eta-}){\rm d}\eta
   +\sum_{i=1}^{n}b_{i}(\eta,\texttt{{\rm{X}}}_{\eta-})\frac{\partial{u}}{\partial x_{i}}(\eta,\texttt{{\rm{X}}}_{\eta-}){\rm d}\eta
   \\
   &+\int_{|\texttt{{\rm{y}}}|\geq 1}\left[u(\eta,\texttt{{\rm{X}}}_{\eta-}+\texttt{{\rm{y}}})
   -u(\eta,\texttt{{\rm{X}}}_{\eta-})\right]N({\rm d}\eta,{\rm d}\texttt{{\rm{y}}})
   \\
   &+\int_{|\texttt{{\rm{y}}}|<1}\left[u(\eta,\texttt{{\rm{X}}}_{\eta-}+\texttt{{\rm{y}}})
   -u(\eta,\texttt{{\rm{X}}}_{\eta-})\right]\widetilde{N}({\rm d}\eta,{\rm d}\texttt{{\rm{y}}})
   \\
   &+\int_{|\texttt{{\rm{y}}}|<1}\left[u(\eta,\texttt{{\rm{X}}}_{\eta-}+\texttt{{\rm{y}}})
   -u(\eta,\texttt{{\rm{X}}}_{\eta-})-(\nabla{u(\eta,\texttt{{\rm{X}}}_{\eta-})},
   \texttt{{\rm{y}}})\right]\nu({\rm d}\texttt{{\rm{y}}}){\rm d}\eta.
  \end{aligned}
  \end{equation}
 Since
 \begin{equation}
 {\rm d}Y_{\eta}=c(\eta,\texttt{{\rm{X}}}_{\eta-})Y_{\eta}{\rm d}\eta
 \quad {\rm and} \quad
 {\rm d}Y_{\eta}\cdot{\rm d}u(\eta,\texttt{{\rm{X}}}_{\eta})=0,
 \end{equation}
 it follows that
 \begin{equation}
  {\rm d}u(\eta,\texttt{{\rm{X}}}_{\eta})Y_{\eta}=u(\eta,\texttt{{\rm{X}}}_{\eta}){\rm d}Y_{\eta}
  +Y_{\eta}{\rm d}u(\eta,\texttt{{\rm{X}}}_{\eta}).
 \end{equation}
 Thus, we obtain
 \begin{equation}
  \begin{aligned}
   &\mathbf{E}\bigg\{
   u(T\wedge\tau_{t,\texttt{{\rm{x}}}},\texttt{{\rm{X}}}_{T\wedge\tau_{t,\texttt{{\rm{x}}}}}
   ^{t,\texttt{{\rm{x}}}})  Y_{T\wedge\tau_{t,\texttt{{\rm{x}}}}}
   ^{t,\texttt{{\rm{x}}},1}+Z_{T\wedge\tau_{t,\texttt{{\rm{x}}}}}
   ^{t,\texttt{{\rm{x}}},1,0}-u(t,\texttt{{\rm{x}}})\bigg\}
   \\
   =&\mathbf{E}\Bigg\{\int_{t}^{T\wedge\tau_{t,\texttt{{\rm{x}}}}}
   \bigg[\frac{\partial{u}}{\partial \eta}(\eta,\texttt{{\rm{X}}}_{\eta-})
   +\sum_{i=1}^{n}b_{i}(\eta,\texttt{{\rm{X}}}_{\eta-})\frac{\partial{u}}{\partial x_{i}}
   (\eta,\texttt{{\rm{X}}}_{\eta-})
   \\
   &+\int_{\texttt{{\rm{y}}}\in\mathbb{R}^{n}}\Big[u(\eta,\texttt{{\rm{X}}}_{\eta-}+\texttt{{\rm{y}}})
   -u(\eta,\texttt{{\rm{X}}}_{\eta-})-I_{\{|\texttt{{\rm{y}}}|<1\}}
   (\nabla{u(\eta,\texttt{{\rm{X}}}_{\eta-})},\texttt{{\rm{y}}})\Big]\nu({\rm d}\texttt{{\rm{y}}})
   \\
   &+f(\eta,\texttt{{\rm{X}}}_{\eta-})
   +c(\eta,\texttt{{\rm{X}}}_{\eta-})u(\eta,\texttt{{\rm{X}}}_{\eta-})\bigg]
   \exp{\left(\int_{t}^{\eta}c(v,\texttt{{\rm{X}}}_{v-}){\rm d}v\right)}
   {\rm d}\eta\Bigg\}
   \\
   =&0,
  \end{aligned}
 \end{equation}
 where we have used
 \begin{equation}
 \int_{|\texttt{{\rm{y}}}|<1}(\nabla{u(\eta,\texttt{{\rm{X}}}_{\eta-})},
 \texttt{{\rm{y}}})\nu({\rm d}\texttt{{\rm{y}}})=0, \quad
 Y_{\eta}=\exp{\left(\int_{t}^{\eta}c(v,\texttt{{\rm{X}}}_{v-}){\rm d}v\right)}
 \end{equation}
 and equation (\ref{eq:terminalboundaryproblem}). Thus we complete the proof.
\end{proof}

At last, we introduce some notations which will be used later.

For $\beta=\lfloor\beta\rfloor+\{\beta\}^{+}>0$, $\lfloor\beta\rfloor\in
\mathbb{Z}^{+}\cup\{0\}$, $\{\beta\}^{+}\in[0,1)$,
let $C^{\beta}([0,T]\times\mathbb{R}^{n})$ denote the space of measurable functions
$u(t,\texttt{{\rm{x}}})$ on $\mathbb{R}^{n}$ for any $t\in[0,T]$
such that the norm
\begin{equation}
 |u(t,\texttt{{\rm{x}}})|_{\beta}=\sum_{|\gamma|\leq \lfloor\beta\rfloor}
  \sup_{(t,\texttt{{\rm{x}}})\in
 [0,T]\times\mathbb{R}^{n}}|\partial_{\texttt{{\rm{x}}}}^{\gamma}u(t,\texttt{{\rm{x}}})|
 +\sup_{\substack{\gamma=\lfloor\beta\rfloor \\t,\texttt{{\rm{x}}},\texttt{{\rm{h}}}\neq \texttt{{\rm{0}}}}}
 \frac{\left|\partial_{\texttt{{\rm{x}}}}^{\gamma}u(t,\texttt{{\rm{x+h}}})
 -\partial_{\texttt{{\rm{x}}}}^{\gamma}u(t,\texttt{{\rm{x}}})\right|}
 {|\texttt{{\rm{h}}}|^{\{\beta\}^{+}}}
\end{equation}
is finite. $C^{\beta}(\mathbb{R}^{n})$ is just the general H\"{o}lder space
whose functions are defined on $\mathbb{R}^{n}$.

Define $C^{1,\beta}([0,T]\times\mathbb{R}^{n})$ as the space of measurable functions
$u(t,\texttt{{\rm{x}}}) \in C^{\beta}([0,T]\times\mathbb{R}^{n})$
whose first-order derivative with respect to $t$ are bounded continuous on
$[0,T]$ for any $\texttt{{\rm{x}}}\in\mathbb{R}^{n}$ with natural norm.
Just like \cite{Deligiannidis2021}, we give the following assumptions.
\\
Assumption I.
There exists a unique solution $u(t,\texttt{{\rm{x}}})$ to
parabolic problem (\ref{eq:terminalboundaryproblem}) such that
\begin{equation}
 u(t,\texttt{{\rm{x}}})\in C^{1,\beta}([0,T]\times\mathbb{R}^{n}).
\end{equation}
\\
Assumption II.
The functions $b_{i}(t,\texttt{{\rm{x}}})\,(i=1,\ldots,n)$, $c(t,\texttt{{\rm{x}}})$,
$f(t,\texttt{{\rm{x}}})$, and their first derivatives with respect to $t$
and $\texttt{{\rm{x}}}$ are uniformly bounded.
\\
Assumption III.
The functions $b_{i}(t,\texttt{{\rm{x}}})\,(i=1,\ldots,n)$, $c(t,\texttt{{\rm{x}}})$,
$f(t,\texttt{{\rm{x}}})$, their first derivatives with respect to
$t$ and $\texttt{{\rm{x}}}$ and their second derivatives with respect to
$\texttt{{\rm{x}}}$ are uniformly bounded.
\section{Jump-adapted Euler scheme}
For $\varepsilon\in(0,1)$, we approximate the symmetric $2s$-stable process
\begin{equation}
L_{t}=\int_{0}^{t}\int_{|\texttt{{\rm{y}}}|<1}\texttt{{\rm{y}}}
  \widetilde{N}({\rm d}\eta,{\rm d}\texttt{{\rm{y}}})+\int_{0}^{t}\int_{|\texttt{{\rm{y}}}|\geq1}
  \texttt{{\rm{y}}} N({\rm d}\eta,{\rm d}\texttt{{\rm{y}}})
\end{equation}
by
\begin{equation}\label{eq:approLevy}
L_{t}^{\varepsilon}=\int_{0}^{t}\int_{\varepsilon<|\texttt{{\rm{y}}}|<1}\texttt{{\rm{y}}}
  \widetilde{N}({\rm d}\eta,{\rm d}\texttt{{\rm{y}}})+\int_{0}^{t}\int_{|\texttt{{\rm{y}}}|\geq1}
  \texttt{{\rm{y}}} N({\rm d}\eta,{\rm d}\texttt{{\rm{y}}}).
\end{equation}
Unlike replacing small jumps by Brownian motion in \cite{Higa&Tankov2010,Mikulevicius2012},
we directly remove small jumps. In the following lemma, we give the weak convergent order.

\begin{lemma}
 Let $s\in(0,1)$, $h\in C^{\beta}(\mathbb{R}^{n})$, $\beta\in (2s\vee1,2)$. Then there is a
 constant $C_{1}$, such that for
 $0\leq t'\leq t\leq T$, we have
 \begin{equation}
  \left|\mathbf{E}\left[h(L_{t-t'})-h(L_{t-t'}^{\varepsilon})\right]\right|
  \leq C_{1}\varepsilon^{\beta-2s}|h|_{\beta}(t-t').
 \end{equation}
\end{lemma}
\begin{proof}
 By It\^{o} formula,
 \begin{equation}
  v(t',\texttt{{\rm{x}}})=\mathbf{E}[h(\texttt{{\rm{x}}}+L_{t}^{\varepsilon}-L_{t'}^{\varepsilon})]
 \end{equation}
 is the solution of the backward Kolmogorov equation with $0\leq \eta\leq t$,
 \begin{equation}\label{eq:Kolmogorovequ}
 \left\{
  \begin{aligned}
   &\frac{\partial v}{\partial \eta}(\eta,\texttt{{\rm{x}}})
   +\int_{|\texttt{{\rm{y}}}|>\varepsilon}v(\eta,\texttt{{\rm{x+y}}})-v(\eta,\texttt{{\rm{x}}})
   -I_{\{|\texttt{{\rm{y}}}|<1\}}(\nabla v(\eta,\texttt{{\rm{x}}}),\texttt{{\rm{y}}})
   \nu({\rm d}\texttt{{\rm{y}}})=0,
   \\[3pt]
   &v(t,\texttt{{\rm{x}}})=h(\texttt{{\rm{x}}}).
  \end{aligned}
  \right.
 \end{equation}
 We can get $v\in C^{\beta}([0,t]\times \mathbb{R}^{n})$ for any $\eta\in[0,t]$
 since $|v|_{\beta}\leq|h|_{\beta}$.
 By It\^{o} formula and (\ref{eq:Kolmogorovequ}), we have
 \begin{equation}
  \begin{aligned}
   &\big|\mathbf{E}\left[h(L_{t}-L_{t'})\right]
   -\mathbf{E}\left[h(L_{t}^{\varepsilon}-L_{t'}^{\varepsilon})\right]\big|
   \\[3pt]
   =&\big|\mathbf{E}[v(t,L_{t}-L_{t'})-v(t',0)]\big|
   \\[3pt]
   =&\,\Bigg|\mathbf{E}\bigg[\int_{t'}^{t}\frac{\partial v}{\partial \eta}(\eta,L_{\eta-}-L_{t'})
   {\rm d}\eta+\int_{t'}^{t}\int_{\texttt{{\rm{y}}}\in\mathbb{R}^{n}}\big[v(\eta,L_{\eta-}-L_{t'}
   +\texttt{{\rm{y}}})
   \\[3pt]
   &-v(\eta,L_{\eta-}-L_{t'})-I_{\{|\texttt{{\rm{y}}}|<1\}}(\nabla v(\eta,L_{\eta-}-L_{t'}),
   \texttt{{\rm{y}}})\big]\nu({\rm d}\texttt{{\rm{y}}})\,{\rm d}\eta\bigg]\Bigg|
   \\[3pt]
   =&\,\Bigg|\mathbf{E}\bigg[\int_{t'}^{t}\int_{|\texttt{{\rm{y}}}|<\varepsilon}
   \big[v(\eta,L_{\eta-}-L_{t'}+\texttt{{\rm{y}}})
   \\[3pt]
   &-v(\eta,L_{\eta-}-L_{t'})-I_{\{|\texttt{{\rm{y}}}|<1\}}(\nabla v(\eta,L_{\eta-}-L_{t'}),
   \texttt{{\rm{y}}})\big]\nu({\rm d}\texttt{{\rm{y}}})\,{\rm d}\eta\bigg]\Bigg|
   \\[3pt]
   \leq& \int_{t'}^{t}\int_{|\texttt{{\rm{y}}}|<\varepsilon}\int_{0}^{1}
   \mathbf{E}\left|(\nabla v(\eta,L_{\eta-}-L_{t'}+\alpha\texttt{{\rm{y}}})
   -\nabla v(\eta,L_{\eta-}-L_{t'}),\texttt{{\rm{y}}})\right|
   {\rm d}\alpha\,\nu({\rm d}\texttt{{\rm{y}}}){\rm d}\eta
   \\[3pt]
   \leq&\, C_{1}(t-t')|v|_{\beta}\int_{|\texttt{{\rm{y}}}|<\varepsilon}
   |\texttt{{\rm{y}}}|^{-n-2s+\beta}{\rm d}\texttt{{\rm{y}}}
   \\[3pt]
   \leq&\, C_{1}(t-t')|h|_{\beta}\,\varepsilon^{\beta-2s}.
  \end{aligned}
 \end{equation}
\end{proof}

Consider the following $L_{t}^{\varepsilon}$-jump adapted time discretization:
\begin{equation}
 \left\{
 \begin{aligned}
 &\tau_{0}=t,
 \\[3pt]
 &\tau_{i+1}=\inf\{t>\tau_{i}:\Delta{L_{t}^{\varepsilon}}\neq 0\}\wedge
 \widetilde{\tau}_{t,\texttt{{\rm{x}}}}\wedge{T},
 \end{aligned}
 \right.
\end{equation}
where $\widetilde{\tau}_{t,\texttt{{\rm{x}}}}=\left\{r\geq{t},
\widetilde{\texttt{{\rm{X}}}}_{r}\notin{\mathbb{D}}\right\}$ and
$\widetilde{\texttt{{\rm{X}}}}_{r}$ is given by the following system.
For $r\in(\tau_{i},\tau_{i+1})$, we define
\begin{equation}\label{eq:approSDEs12}
 \left\{
 \begin{aligned}
  &\widetilde{\texttt{{\rm{X}}}}_{r}=\widetilde{\texttt{{\rm{X}}}}_{\tau_{i}}
  +b(\tau_{i},\widetilde{\texttt{{\rm{X}}}}_{\tau_{i}})(r-\tau_{i}),
  \\[3pt]
  &\widetilde{Y}_{r}=\widetilde{Y}_{\tau_{i}}
  +c(\tau_{i},\widetilde{\texttt{{\rm{X}}}}_{\tau_{i}})
  \widetilde{Y}_{\tau_{i}}(r-\tau_{i}),
  \\[3pt]
  &\widetilde{Z}_{r}=\widetilde{Z}_{\tau_{i}}
  +f(\tau_{i},\widetilde{\texttt{{\rm{X}}}}_{\tau_{i}})
  \widetilde{Y}_{\tau_{i}}(r-\tau_{i}),
 \end{aligned}
 \right.
\end{equation}
where the starting point $\widetilde{\texttt{{\rm{X}}}}_{t}=\texttt{{\rm{x}}}$,
$\widetilde{Y}_{t}=1$, and $\widetilde{Z}_{t}=0$. For $r=\tau_{i+1}$, we define
\begin{equation}\label{eq:approSDEs}
 \left\{
 \begin{aligned}
  &\widetilde{\texttt{{\rm{X}}}}_{\tau_{i+1}}=\widetilde{\texttt{{\rm{X}}}}_{\tau_{i+1}{-}}
  +\Delta L_{\tau_{i+1}}^{\varepsilon},
  \quad {\rm where}\quad \Delta L_{\tau_{i+1}}^{\varepsilon}=L_{\tau_{i+1}}^{\varepsilon}
  -L_{\tau_{i+1}^{-}}^{\varepsilon},
  \\[3pt]
  &\widetilde{Y}_{\tau_{i+1}}=\widetilde{Y}_{\tau_{i}}
  +c(\tau_{i},\widetilde{\texttt{{\rm{X}}}}_{\tau_{i}})
  \widetilde{Y}_{\tau_{i}}(\tau_{i+1}-\tau_{i}),
  \\[3pt]
  &\widetilde{Z}_{\tau_{i+1}}=\widetilde{Z}_{\tau_{i}}
  +f(\tau_{i},\widetilde{\texttt{{\rm{X}}}}_{\tau_{i}})
  \widetilde{Y}_{\tau_{i}}(\tau_{i+1}-\tau_{i}).
 \end{aligned}
 \right.
\end{equation}
 Thus, we approximate the solution $u(t,\texttt{{\rm{x}}})$ of equation
 (\ref{eq:terminalboundaryproblem})
by
\begin{equation}\label{eq:approSolution}
 \mathbf{E}\left[
   u(T\wedge\widetilde{\tau}_{t,\texttt{{\rm{x}}}},
   \widetilde{\texttt{{\rm{X}}}}_{T\wedge\widetilde{\tau}_{t,\texttt{{\rm{x}}}}})
   \widetilde{Y}_{T\wedge\widetilde{\tau}_{t,\texttt{{\rm{x}}}}}
   +\widetilde{Z}_{T\wedge\widetilde{\tau}_{t,\texttt{{\rm{x}}}}}\right].
\end{equation}

We summarize the method in Algorithm 1.
\begin{algorithm}[t]\label{algo:sde}
\caption{Algorithm for (\ref{eq:approSolution})} 
\begin{algorithmic}[1]
\State
{\bf initialize:}
$\tau_{0}=t$, $\widetilde{\texttt{{\rm{X}}}}_{\tau_{0}}=\texttt{{\rm{x}}}$,
$\widetilde{Y}_{\tau_{0}}=1$, and $\widetilde{Z}_{\tau_{0}}=0$, $i=0$.
\While{$\tau_{i}<T$ and $\widetilde{\texttt{{\rm{X}}}}_{\tau_{i}}\in D$} 
　　\State {\bf Sample:} jump time $\tau$ ($\tau$ is exponentially distributed with parameter
    $\lambda_{\varepsilon}$ in (\ref{eq:jump intensity})).
    \\
\hspace*{0.18in} {\bf Set:} $\widetilde{\texttt{{\rm{X}}}}_{\tau_{i+1}^{-}}=\widetilde{\texttt{{\rm{X}}}}_{\tau_{i}}
  +b(\tau_{i},\widetilde{\texttt{{\rm{X}}}}_{\tau_{i}})\tau$.
     \If{$\tau+\tau_{i}>T$ or $\widetilde{\texttt{{\rm{X}}}}_{\tau_{i+1}^{-}}\notin \overline{D}$},
　　　　 \State {\bf Set:} $\tau_{i+1}=\sup\left\{\tau_{i}<t<T:\widetilde{\texttt{{\rm{X}}}}_{\tau_{i}}
           +b(\tau_{i},\widetilde{\texttt{{\rm{X}}}}_{\tau_{i}})(t-\tau_{i})\in \overline{D}\right\}$.
       \State {\bf Evaluate:} $\widetilde{\texttt{{\rm{X}}}}_{\tau_{i+1}}$ without jump,
         $\widetilde{Y}_{\tau_{i+1}}$, $\widetilde{Z}_{\tau_{i+1}}$ according to (\ref{eq:approSDEs}).
       \State{\bf Set:} $i=i+1$.
       \Else
　　　　\State {\bf Sample:} jump size $J_{\varepsilon}$ according to density
        (\ref{eq:PossionRandomMeasure}).
         \State {\bf Evaluate:} $\widetilde{\texttt{{\rm{X}}}}_{\tau_{i+1}}$,
         $\widetilde{Y}_{\tau_{i+1}}$, $\widetilde{Z}_{\tau_{i+1}}$ according to (\ref{eq:approSDEs}).
         \State{\bf Set:} $\tau_{i+1}=\tau_{i}+\tau$ and $i=i+1$.
   \EndIf \State {\bf end if}
\EndWhile \State {\bf end while}
\If{$\tau_{i+1}<T$} {\bf Set:} $\widetilde\tau_{t,\texttt{{\rm{x}}}}=\tau_{i+1}$,
 {\bf Evaluate:} $\chi(\widetilde{\tau}_{t,\texttt{{\rm{x}}}},
   \widetilde{\texttt{{\rm{X}}}}_{\widetilde{\tau}_{t,\texttt{{\rm{x}}}}})
   \widetilde{Y}_{\widetilde{\tau}_{t,\texttt{{\rm{x}}}}}
   +\widetilde{Z}_{\widetilde{\tau}_{t,\texttt{{\rm{x}}}}}$.
    \Else \quad
    {\bf Set:} $\widetilde\tau_{t,\texttt{{\rm{x}}}}=T$,
 {\bf Evaluate:} $g(T,\widetilde{\texttt{{\rm{X}}}}_{T})\widetilde{Y}_{T}+\widetilde{Z}_{T}$.
\EndIf \State {\bf end if}
\State Loop above algorithm $N$ times.
\State{\bf Evaluate:}
$u(t,\texttt{{\rm{x}}})\approx \frac{1}{N}
\sum\limits_{j=1}^{N}\left[u(T\wedge\widetilde{\tau}_{t,\texttt{{\rm{x}}}},
   \widetilde{\texttt{{\rm{X}}}}_{T\wedge\widetilde{\tau}_{t,\texttt{{\rm{x}}}}}^{j})
   \widetilde{Y}_{T\wedge\widetilde{\tau}_{t,\texttt{{\rm{x}}}}}^{j}
   +\widetilde{Z}_{T\wedge\widetilde{\tau}_{t,\texttt{{\rm{x}}}}}^{j}\right]$.
\end{algorithmic}
\end{algorithm}

To give the error estimate, we need the following lemmas.
\begin{lemma} \label{lemma:jumptime}
 Let $\delta_{i}=(\widetilde{\tau}_{t,\texttt{{\rm{x}}}}\wedge T)-\tau_{i}$, and
 \begin{equation}\label{eq:jump intensity}
  \lambda_{\varepsilon}=\nu(|\texttt{{\rm{y}}}|>\varepsilon)
  =\frac{2^{2s}\Gamma\left(\frac{n}{2}+s\right)\varepsilon^{-2s}}
  {\Gamma(1-s)\Gamma{\left(\frac{n}{2}\right)}}
 \end{equation}
 be the jump intensity. Then
 \\
 {\rm(i)}there is a constant $C_{1}$, such that for any $i\geq0$,
 \begin{equation}
  C_{1}(\delta_{i}\wedge\lambda_{\varepsilon}^{-1})
  \leq\mathbf{E}\left[\tau_{i+1}-\tau_{i}|\mathcal{F}_{\tau_{i}}\right]
  \leq \delta_{i}\wedge\lambda_{\varepsilon}^{-1},
 \end{equation}
 \\
 {\rm(ii)}there is a constant $C_{2}$, such that for any $i\geq0$,
 \begin{equation}
  \begin{aligned}
  \mathbf{E}\left[(\tau_{i+1}-\tau_{i})^{\frac{3}{2}}
  |\mathcal{F}_{\tau_{i}}\right]
  &\leq C_{2}\left(\delta_{i}^{\frac{1}{2}}\wedge
  \lambda_{\varepsilon}^{-\frac{1}{2}}\right)\mathbf{E}
  \left[\tau_{i+1}-\tau_{i}|\mathcal{F}_{\tau_{i}}\right]
  ,
  \end{aligned}
 \end{equation}
 \\
 {\rm(iii)}there is a constant $C_{3}$, such that for any $i\geq0$,
 \begin{equation}
  \begin{aligned}
  \mathbf{E}\left[(\tau_{i+1}-\tau_{i})^2|\mathcal{F}_{\tau_{i}}\right]
  &\leq C_{3}\mathbf{E}\left[\delta_{i}^{2}\wedge\lambda_{\varepsilon}^{-2}|\mathcal{F}_{\tau_{i}}\right]
  \\[3pt]
  &\leq C_{3}(\delta_{i}\wedge\lambda_{\varepsilon}^{-1})\mathbf{E}
  \left[\tau_{i+1}-\tau_{i}|\mathcal{F}_{\tau_{i}}\right]
  ,
  \end{aligned}
 \end{equation}
\end{lemma}
\begin{proof}
 The proofs of (i) and (iii) are similar to those of Lemma 8 in \cite{Mikulevicius2012}.
 So we omit them. For (ii), we have
 \begin{equation}
  \begin{aligned}
   &\mathbf{E}\left[(\tau_{i+1}-\tau_{i})^\frac{3}{2}|\mathcal{F}_{\tau_{i}}\right]
  =\lambda_{\varepsilon}\int_{0}^{\delta_{i}}e^{-\lambda_{\varepsilon}t}t^{\frac{3}{2}}
  {\rm d}t+e^{-\lambda_{\varepsilon}\delta_{i}}{\delta_{i}}^{\frac{3}{2}}
  \\[3pt]
  \leq& \,C_{3}(\delta_{i}^{\frac{3}{2}}\wedge\lambda_{\varepsilon}^{-\frac{3}{2}})
  \leq C_{3}(\delta_{i}^{\frac{1}{2}}\wedge\lambda_{\varepsilon}^{-\frac{1}{2}})
  \mathbf{E}\left[\tau_{i+1}-\tau_{i}|\mathcal{F}_{\tau_{i}}\right].
  \end{aligned}
 \end{equation}
\end{proof}

We also need the following auxiliary lemma.
\begin{lemma}\label{lemma:Ybound}
 Under the Assumption {\rm II}, $\widetilde{Y}_{\tau_{i}}$ defined in {\rm (\ref{eq:approSDEs})}
 is uniformly bounded by a deterministic formula
 \begin{equation}
  |\widetilde{Y}_{\tau_{i}}|<e^{\overline{c}(T-t)}, \quad i=0,1,\ldots,
  n_{T\wedge{\widetilde{\tau}_{t,\texttt{{\rm{x}}}}}},
 \end{equation}
 where $\overline{c}=\max\limits_{(t,\texttt{{\rm{x}}})\in[0,T]\times \mathbb{D}}c(t,\texttt{{\rm{x}}})$.
\end{lemma}
\begin{proof}
 From (\ref{eq:approSDEs}) and $\widetilde{Y}_{\tau_{0}}=\widetilde{Y}_{t}=1$, we get the required
 estimate as follows,
 \begin{equation}
  \begin{aligned}
 |\widetilde{Y}_{\tau_{i}}|=&|\widetilde{Y}_{\tau_{i-1}}
  +c(\tau_{i-1},\widetilde{\texttt{{\rm{X}}}}_{\tau_{i-1}})
  \widetilde{Y}_{\tau_{i-1}}(\tau_{i}-\tau_{i-1})|
  \\[3pt]
  &\leq |\widetilde{Y}_{\tau_{i-1}}||1+\overline{c}(\tau_{i}-\tau_{i-1})|
  \leq |\widetilde{Y}_{\tau_{0}}|e^{\overline{c}(\tau_{i}-t)}
  \leq e^{\overline{c}(T-t)}.
  \end{aligned}
 \end{equation}
\end{proof}

Now we prove the convergence theorem for Algorithm 1.
\begin{theorem}\label{thm:Errorestimate}
 Let $\varepsilon\in(0,1)$. Under Assumption {\rm I} with $\beta\in(2s\vee1,2)$
 and Assumption {\rm II}, the error estimate
 \begin{equation}
  \left|\mathbf{E}\left[
   u(T\wedge\widetilde{\tau}_{t,\texttt{{\rm{x}}}},
   \widetilde{\texttt{{\rm{X}}}}_{T\wedge\widetilde{\tau}_{t,\texttt{{\rm{x}}}}})
   \widetilde{Y}_{T\wedge\widetilde{\tau}_{t,\texttt{{\rm{x}}}}}
   +\widetilde{Z}_{T\wedge\widetilde{\tau}_{t,\texttt{{\rm{x}}}}}\right]
    -u(t,\texttt{{\rm{x}}})\right|
    \leq C_{1}\varepsilon^{2s}+C_{2}\varepsilon^{\beta-2s}
 \end{equation}
 holds with $C_{1}$ and $C_{2}$ being constants independent of $t$ and $\texttt{{\rm{x}}}$.
\end{theorem}
\begin{proof}
 \begin{equation}\label{eq:errorestimate1}
  \begin{aligned}
  &\left|\mathbf{E}\left[
   u(T\wedge\widetilde{\tau}_{t,\texttt{{\rm{x}}}},
   \widetilde{\texttt{{\rm{X}}}}_{T\wedge\widetilde{\tau}_{t,\texttt{{\rm{x}}}}})
   \widetilde{Y}_{T\wedge\widetilde{\tau}_{t,\texttt{{\rm{x}}}}}
   +\widetilde{Z}_{T\wedge\widetilde{\tau}_{t,\texttt{{\rm{x}}}}}\right]
    -u(t,\texttt{{\rm{x}}})\right|
    \\[3pt]
   =& \left|\mathbf{E}\left\{\sum_{i=0}^{n_{T\wedge\widetilde{\tau}_{t,\texttt{{\rm{x}}}}}-1}
   \left[u(\tau_{i+1},\widetilde{\texttt{{\rm{X}}}}_{\tau_{i+1}})\widetilde{Y}_{\tau_{i+1}}
   +\widetilde{Z}_{\tau_{i+1}}-u(\tau_{i},\widetilde{\texttt{{\rm{X}}}}_{\tau_{i}})
   \widetilde{Y}_{\tau_{i}}-\widetilde{Z}_{\tau_{i}}\right]\right\}\right|.
  \end{aligned}
 \end{equation}
 By It\^{o} formula, the martingale property, systems (\ref{eq:approSDEs12}) and
 (\ref{eq:approSDEs}), we have
  \begin{align}
  &\left|\mathbf{E}\left[
   u(T\wedge\widetilde{\tau}_{t,\texttt{{\rm{x}}}},
   \widetilde{\texttt{{\rm{X}}}}_{T\wedge\widetilde{\tau}_{t,\texttt{{\rm{x}}}}})
   \widetilde{Y}_{T\wedge\widetilde{\tau}_{t,\texttt{{\rm{x}}}}}
   +\widetilde{Z}_{T\wedge\widetilde{\tau}_{t,\texttt{{\rm{x}}}}}\right]
    -u(t,\texttt{{\rm{x}}})\right|
    \nonumber\\[3pt]
   =& \Bigg|\mathbf{E}\Bigg\{\sum_{i=0}^{n_{T\wedge\widetilde{\tau}_{t,\texttt{{\rm{x}}}}}-1}
   u(\tau_{i+1},\widetilde{\texttt{{\rm{X}}}}_{\tau_{i+1}{-}})\widetilde{Y}_{\tau_{i+1}}
   +\widetilde{Z}_{\tau_{i+1}}-u(\tau_{i},\widetilde{\texttt{{\rm{X}}}}_{\tau_{i}})
   \widetilde{Y}_{\tau_{i}}-\widetilde{Z}_{\tau_{i}}
   \nonumber\\[3pt]
   &+\widetilde{Y}_{\tau_{i+1}}\Big[u(\tau_{i+1},\widetilde{\texttt{{\rm{X}}}}_{\tau_{i+1}})
   -u(\tau_{i+1},\widetilde{\texttt{{\rm{X}}}}_{\tau_{i+1}{-}})\Big]\Bigg\}\Bigg|
   \nonumber\\[3pt]
   =& \Bigg|\mathbf{E}\Bigg\{\sum_{i=0}^{n_{T\wedge\widetilde{\tau}_{t,\texttt{{\rm{x}}}}}-1}
   \widetilde{Y}_{\tau_{i}}\left[u(\tau_{i+1},\widetilde{\texttt{{\rm{X}}}}_{\tau_{i+1}{-}})
   -u(\tau_{i},\widetilde{\texttt{{\rm{X}}}}_{\tau_{i}})
   +c(\tau_{i},\widetilde{\texttt{{\rm{X}}}}_{\tau_{i}})
   u(\tau_{i+1},\widetilde{\texttt{{\rm{X}}}}_{\tau_{i+1}{-}})\Delta{\tau_{i}}\right]
   \nonumber\\[3pt]
   &+\widetilde{Y}_{\tau_{i+1}}\Big[u(\tau_{i+1},\widetilde{\texttt{{\rm{X}}}}_{\tau_{i+1}})
   -u(\tau_{i+1},\widetilde{\texttt{{\rm{X}}}}_{\tau_{i+1}{-}})\Big]
   +f(\tau_{i},\widetilde{\texttt{{\rm{X}}}}_{\tau_{i}})\Delta{\tau_{i}}\Bigg\}\Bigg|
   \nonumber\\[3pt]
   =& \Bigg|\mathbf{E}\Bigg[\sum_{i=0}^{n_{T\wedge\widetilde{\tau}_{t,\texttt{{\rm{x}}}}}-1}
     \widetilde{Y}_{\tau_{i}}\int_{\tau_{i}}^{\tau_{i+1}}\frac{\partial}{\partial{\eta}}
   u(\eta,\widetilde{\texttt{{\rm{X}}}}_{\eta-})
   +\sum_{j=1}^{n}b_{j}(\tau_{i},\widetilde{\texttt{{\rm{X}}}}_{\tau_{i}})
   \frac{\partial}{\partial{x}_{j}}u(\eta,\widetilde{\texttt{{\rm{X}}}}_{\eta-})
   {\rm d}\eta
   \nonumber\\[3pt]
   &+\widetilde{Y}_{\tau_{i}}c(\tau_{i},\widetilde{\texttt{{\rm{X}}}}_{\tau_{i}})
   u(\tau_{i+1},\widetilde{\texttt{{\rm{X}}}}_{\tau_{i+1}{-}})\Delta{\tau_{i}}
   +\widetilde{Y}_{\tau_{i}}f(\tau_{i},\widetilde{\texttt{{\rm{X}}}}_{\tau_{i}})\Delta{\tau_{i}}\Bigg]
   \nonumber\\[3pt]
   &+\mathbf{E}\left[\int_{t}^{T\wedge\widetilde{\tau}_{t,\texttt{{\rm{x}}}}}
   \widetilde{Y}_{\eta}\int_{|\texttt{{\rm{y}}}|>\varepsilon}\left[
   u(\eta,\widetilde{\texttt{{\rm{X}}}}_{\eta-}+\texttt{{\rm{y}}})
   -u(\eta,\widetilde{\texttt{{\rm{X}}}}_{\eta-})\right]N({\rm d}\eta,{\rm d}\texttt{{\rm{y}}})
   \right]\Bigg|
   \nonumber\\[3pt]
   =&\Bigg|\mathbf{E}\Bigg\{\sum_{i=0}^{n_{T\wedge\widetilde{\tau}_{t,\texttt{{\rm{x}}}}}-1}
     \widetilde{Y}_{\tau_{i}}\mathbf{E}\Bigg[\int_{\tau_{i}}^{\tau_{i+1}}\frac{\partial}{\partial{\eta}}
   u(\eta,\widetilde{\texttt{{\rm{X}}}}_{\eta-})
    +\sum_{j=1}^{n}b_{j}(\tau_{i},\widetilde{\texttt{{\rm{X}}}}_{\tau_{i}})
   \frac{\partial}{\partial{x}_{j}}u(\eta,\widetilde{\texttt{{\rm{X}}}}_{\eta-}){\rm d}\eta
   \nonumber\\[3pt]
   &+\int_{\tau_{i}}^{\tau_{i+1}}\int_{|\texttt{{\rm{y}}}|>\varepsilon}\left[
   u(\eta,\widetilde{\texttt{{\rm{X}}}}_{\eta-}+\texttt{{\rm{y}}})
   -u(\eta,\widetilde{\texttt{{\rm{X}}}}_{\eta-})\right]
   \nu({\rm d}\texttt{{\rm{y}}}){\rm d}\eta
   \nonumber\\[3pt]
   &+c(\tau_{i},\widetilde{\texttt{{\rm{X}}}}_{\tau_{i}})\Delta{\tau_{i}}
   \int_{\tau_{i}}^{\tau_{i+1}}\int_{|\texttt{{\rm{y}}}|>\varepsilon}\left[
   u(\eta,\widetilde{\texttt{{\rm{X}}}}_{\eta-}+\texttt{{\rm{y}}})
   -u(\eta,\widetilde{\texttt{{\rm{X}}}}_{\eta-})\right]
   \nu({\rm d}\texttt{{\rm{y}}}){\rm d}\eta
   \nonumber\\[3pt]
   &+c(\tau_{i},\widetilde{\texttt{{\rm{X}}}}_{\tau_{i}})
   u(\tau_{i+1},\widetilde{\texttt{{\rm{X}}}}_{\tau_{i+1}{-}})\Delta{\tau_{i}}
   +f(\tau_{i},\widetilde{\texttt{{\rm{X}}}}_{\tau_{i}})\Delta{\tau_{i}}
   \Big{|}\mathcal{F}_{\tau_{i}}\Bigg]\Bigg\}\Bigg|
   \nonumber\\[3pt]
   =& \left|\mathbf{E}\left\{\sum_{i}\widetilde{Y}_{\tau_{i}}\mathbf{E}\left[
   B_{i}\big{|}\mathcal{F}_{\tau_{i}}\right]\right\}\right|,
\end{align}
 where $\Delta\tau_{i}=\tau_{i+1}-\tau_{i}$.

 Since $u(t,\texttt{{\rm{x}}})$ satisfies (\ref{eq:terminalboundaryproblem}), we obtain
 \begin{equation}\label{eq:errorestimate2}
  \begin{aligned}
   &\mathbf{E}\left[B_{i}\big{|}\mathcal{F}_{\tau_{i}}\right]
   \\[3pt]
   =&\,\mathbf{E}\bigg\{
   \int_{\tau_{i}}^{\tau_{i+1}}\sum_{j=1}^{n}
   \left[b_{j}(\tau_{i},\widetilde{\texttt{{\rm{X}}}}_{\tau_{i}})
   -b_{j}(\eta,\widetilde{\texttt{{\rm{X}}}}_{\eta-})\right]
   \frac{\partial u}{\partial{x_{j}}}(\eta,\widetilde{\texttt{{\rm{X}}}}_{\eta-}){\rm d}\eta
   \\[3pt]
   &-\int_{\tau_{i}}^{\tau_{i+1}}\int_{|\texttt{{\rm{y}}}|<\varepsilon}
   \left[u(\eta,\widetilde{\texttt{{\rm{X}}}}_{\eta-}
   +\texttt{{\rm{y}}})-u(\eta,\widetilde{\texttt{{\rm{X}}}}_{\eta-})
   -I_{|\texttt{{\rm{y}}}|<1}(\nabla{
   u(\eta,\widetilde{\texttt{{\rm{X}}}}_{\eta-})},\texttt{{\rm{y}}})\right]
   \nu({\rm d}\texttt{{\rm{y}}}){\rm d}\eta
   \\[3pt]
   &+c(\tau_{i},\widetilde{\texttt{{\rm{X}}}}_{\tau_{i}})\Delta{\tau_{i}}
   \int_{\tau_{i}}^{\tau_{i+1}}\int_{|\texttt{{\rm{y}}}|>\varepsilon}\left[
   u(\eta,\widetilde{\texttt{{\rm{X}}}}_{\eta-}+\texttt{{\rm{y}}})
   -u(\eta,\widetilde{\texttt{{\rm{X}}}}_{\eta-})\right]
   \nu({\rm d}\texttt{{\rm{y}}}){\rm d}\eta
   \\[3pt]
   &+\int_{\tau_{i}}^{\tau_{i+1}}c(\tau_{i},\widetilde{\texttt{{\rm{X}}}}_{\tau_{i}})
   u(\tau_{i+1},\widetilde{\texttt{{\rm{X}}}}_{\tau_{i+1}-})
   -c(\eta,\widetilde{\texttt{{\rm{X}}}}_{\eta-})
   u(\eta,\widetilde{\texttt{{\rm{X}}}}_{\eta-}){\rm d}\eta
   \\[3pt]
   &+\int_{\tau_{i}}^{\tau_{i+1}}f(\tau_{i},\widetilde{\texttt{{\rm{X}}}}_{\tau_{i}})
   -f(\eta,\widetilde{\texttt{{\rm{X}}}}_{\eta-}){\rm d}\eta
    \Big{|}\mathcal{F}_{\tau_{i}}\bigg\}
    \\[3pt]
   =&\,\widetilde{Y}_{\tau_{i}}\mathbf{E}\left[B_{i1}+B_{i2}+B_{i3}+B_{i4}+B_{i5}
   \big{|}\mathcal{F}_{\tau_{i}}\right].
  \end{aligned}
 \end{equation}
 For the first term, one has
 \begin{equation}\label{eq:B1}
  \begin{aligned}
  &\left|\mathbf{E}\left[B_{i1}
   \big{|}\mathcal{F}_{\tau_{i}}\right]\right|
   \\[3pt]
  \leq &\, \mathbf{E}\left[
  \sum_{j=1}^{n}\int_{\tau_{i}}^{\tau_{i+1}}
  \left|b_{j}(\tau_{i},\widetilde{\texttt{{\rm{X}}}}_{\tau_{i}})
  -b_{j}(\eta,\widetilde{\texttt{{\rm{X}}}}_{\eta-})\right|
  \left|\frac{\partial u}{\partial{x_{j}}}(\eta,\widetilde{\texttt{{\rm{X}}}}_{\eta-})\right|
  {\rm d}\eta\Big{|}\mathcal{F}_{\tau_{i}}\right]
  \\[3pt]
  \leq& \,|u|_{\beta}\mathbf{E}\left[\sum_{j=1}^{n}\int_{\tau_{i}}^{\tau_{i+1}}
  \left|b_{j}(\tau_{i},\widetilde{\texttt{{\rm{X}}}}_{\tau_{i}})
  -b_{j}(\eta,\widetilde{\texttt{{\rm{X}}}}_{\tau_{i}})\right|
  +\left|b_{j}(\eta,\widetilde{\texttt{{\rm{X}}}}_{\tau_{i}})
  -b_{j}(\eta,\widetilde{\texttt{{\rm{X}}}}_{\eta-})\right|{\rm d}\eta
  \Big{|}\mathcal{F}_{\tau_{i}}\right]
  \\[3pt]
  \leq & \,|u|_{\beta}\mathbf{E}\left[\sum_{j=1}^{n}
  C_{1}\int_{\tau_{i}}^{\tau_{i+1}}(\eta-\tau_{i})\,{\rm d}\eta
  +C_{2}\int_{\tau_{i}}^{\tau_{i+1}}\left|\widetilde{\texttt{{\rm{X}}}}_{\tau_{i}}
  -\widetilde{\texttt{{\rm{X}}}}_{\eta-}\right|{\rm d}\eta
  \Big{|}\mathcal{F}_{\tau_{i}}\right]
  \\
  \leq & \,C_{1}|u|_{\beta}\mathbf{E}
  \left[\int_{\tau_{i}}^{\tau_{i+1}}(\eta-\tau_{i})\,{\rm d}\eta
  \Big{|}\mathcal{F}_{\tau_{i}}\right]
  \\[3pt]
  \leq &\, C_{1} \mathbf{E}\left[(\tau_{i+1}-\tau_{i})^2
  \Big{|}\mathcal{F}_{\tau_{i}}\right].
  \end{aligned}
 \end{equation}

 Clearly, for the second term, it holds that
 \begin{equation}\label{eq:B2}
  \begin{aligned}
  &\left|\mathbf{E}\left[B_{i2}\big{|}\mathcal{F}_{\tau_{i}}\right]\right|
   \\[3pt]
  \leq&\, \mathbf{E}\left[\int_{\tau_{i}}^{\tau_{i+1}}\int_{|\texttt{{\rm{y}}}|<\varepsilon}\int_{0}^{1}
  \left|\left(\nabla{u(\eta,\widetilde{\texttt{{\rm{X}}}}_{\eta}+\alpha\texttt{{\rm{y}}}})
  -\nabla{u(\eta,\widetilde{\texttt{{\rm{X}}}}_{\eta})},\texttt{{\rm{y}}}\right)\right|{\rm d}\alpha\,
  \nu({\rm d}\texttt{{\rm{y}}})\,{\rm d}\eta
  \Big{|}\mathcal{F}_{\tau_{i}}\right]
  \\[3pt]
  \leq&\, C_{2} \mathbf{E}\left[(\tau_{i+1}-\tau_{i})\Big{|}\mathcal{F}_{\tau_{i}}
  \right] |u|_{\beta}\int_{|\texttt{{\rm{y}}}|<\varepsilon} |\texttt{{\rm{y}}}|^{-n-2s+\beta}
  {\rm d}\texttt{{\rm{y}}}
  \\[3pt]
  \leq&\,  C_{2} \mathbf{E}\left[\tau_{i+1}-\tau_{i}\Big{|}\mathcal{F}_{\tau_{i}}
  \right]\varepsilon^{\beta-2s}.
  \end{aligned}
 \end{equation}

 Next, we estimate the forth term.
 \begin{equation}\label{eq:B4}
  \begin{aligned}
   &\left|\mathbf{E}\left[B_{i4}\big{|}\mathcal{F}_{\tau_{i}}\right]\right|
   \\[3pt]
   \leq&\left|\mathbf{E}\left[\int_{\tau_{i}}^{\tau_{i+1}}
   c(\tau_{i},\widetilde{\texttt{{\rm{X}}}}_{\tau_{i}})u(\tau_{i+1},\widetilde{\texttt{{\rm{X}}}}_{\tau_{i+1}-})
   -c(\eta,\widetilde{\texttt{{\rm{X}}}}_{\eta})u(\eta,\widetilde{\texttt{{\rm{X}}}}_{\eta-})
   {\rm d}\eta\Big{|}\mathcal{F}_{\tau_{i}}\right]\right|
   \\
   \leq&\bigg|\mathbf{E}\bigg[\int_{\tau_{i}}^{\tau_{i+1}}
   \left[c(\tau_{i},\widetilde{\texttt{{\rm{X}}}}_{\tau_{i}})
   -c(\eta,\widetilde{\texttt{{\rm{X}}}}_{\eta-})\right]
   u(\eta,\widetilde{\texttt{{\rm{X}}}}_{\eta})
   \\[3pt]
   &+\left[u(\tau_{i+1},\widetilde{\texttt{{\rm{X}}}}_{\tau_{i+1}-})
   -u(\eta,\widetilde{\texttt{{\rm{X}}}}_{\eta-})\right]
   c(\tau_{i},\widetilde{\texttt{{\rm{X}}}}_{\tau_{i}}){\rm d}\eta
   \Big{|}\mathcal{F}_{\tau_{i}}\bigg]\bigg|
   \\[3pt]
   \leq&C_{1}\mathbf{E}\left[(\tau_{i+1}-\tau_{i})^2
   \big{|}\mathcal{F}_{\tau_{i}}\right].
  \end{aligned}
 \end{equation}
For the remaining part, we can easily get
\begin{equation}\label{eq:B35}
 \left|\mathbf{E}\left[B_{i3}+B_{i5}\big{|}\mathcal{F}_{\tau_{i}}\right]\right|
 \leq C_{1}\mathbf{E}\left[(\tau_{i+1}-\tau_{i})^2
   \big{|}\mathcal{F}_{\tau_{i}}\right].
\end{equation}
Combining inequalities
(\ref{eq:errorestimate1})-(\ref{eq:B35}) yields
 \begin{equation}
  \begin{aligned}
   &\left|\mathbf{E}\left[u(T\wedge\widetilde{\tau}_{t,\texttt{{\rm{x}}}},
   \widetilde{\texttt{{\rm{X}}}}_{T\wedge\widetilde{\tau}_{t,\texttt{{\rm{x}}}}})
   \widetilde{Y}_{T\wedge\widetilde{\tau}_{t,\texttt{{\rm{x}}}}}
   +\widetilde{Z}_{T\wedge\widetilde{\tau}_{t,\texttt{{\rm{x}}}}}\right]
    -u(t,\texttt{{\rm{x}}})\right|
    \\[3pt]
   \leq&\,\mathbf{E}\left\{\sum_{i}|\widetilde{Y}_{\tau_{i}}|\mathbf{E}\left[\left|
   B_{i1}\right|+\left|B_{i2}\right|+\left|B_{i3}\right|+\left|B_{i4}\right|
   \big{|}\mathcal{F}_{\tau_{i}}\right]\right\}.
   \\
   \leq&\,\mathbf{E}\left\{\sup_{i}|\widetilde{Y}_{\tau_{i}}|\left(\sum_{i}
   \mathbf{E}\left[C_{1}(\tau_{i+1}-\tau_{i})^2
   +C_{2}\varepsilon^{\beta-2s}(\tau_{i+1}-\tau_{i})
   \big{|}\mathcal{F}_{\tau_{i}}\right]
   \right)\right\}
   \\
   \leq&\, C_{1}\varepsilon^{2s}+C_{2}\varepsilon^{\beta-2s}.
  \end{aligned}
 \end{equation}
Thus, the proof is completed.
\end{proof}

\begin{remark}
 The probabilistic approach can be also applied to the Dirichlet problem for
 the steady-state equation
 \begin{equation}\label{eq:boundaryproblem}
 \left\{
 \begin{aligned}
  &-(-\Delta)^su+\sum_{i=1}^{n}b_{i}(\texttt{{\rm{x}}})
  \frac{\partial u}{\partial x_{i}}+c(\texttt{{\rm{x}}})u+
  f(\texttt{{\rm{x}}})=0, \quad &&\texttt{{\rm{x}}}\in \mathbb{D},
  \\[3pt]
  &u(\texttt{{\rm{x}}})=g(\texttt{{\rm{x}}}), \quad &&\texttt{{\rm{x}}}\in
  \mathbb{R}^{n}\setminus\mathbb{D}.
 \end{aligned}
 \right.
\end{equation}
The solution has the probabilistic representation below
 \begin{equation}\label{eq:stableSde}
   u(\texttt{{\rm{x}}})
   =\mathbf{E}\left[
   g(\texttt{{\rm{X}}}_{\tau_{\texttt{{\rm{x}}}}}^{\texttt{{\rm{x}}}})
   Y_{\tau_{\texttt{{\rm{x}}}}}^{\texttt{{\rm{x}}},1}
   +Z_{\tau_{\texttt{{\rm{x}}}}}^{\texttt{{\rm{x}}},1,0}\right],
 \end{equation}
 where $\texttt{{\rm{X}}}_{r}^{\texttt{{\rm{x}}}}$,
 $Y_{r}^{\texttt{{\rm{x}}},1}$, and $Z_{r}^{\texttt{{\rm{x}}},1,0}\,(r\geq 0)$
  are the solution of the Cauchy problem for the
 systems of SDEs,
 \begin{equation}\label{eq:SDEs2}
 \left\{
 \begin{aligned}
  &{\rm d} \texttt{{\rm{X}}}_{r}=b(\texttt{{\rm{X}}}_{r}){\rm d}r
  +{\rm d}L_{r},
  \qquad &&\texttt{{\rm{X}}}_{0}=\texttt{{\rm{x}}},
  \\[3pt]
  &{\rm d}Y_{r}=c(\texttt{{\rm{X}}}_{r})Y_{r}{\rm d}r, \qquad &&Y_{0}=1,
  \\[3pt]
  &{\rm d}Z_{r}=f(\texttt{{\rm{X}}}_{r})Y_{r}{\rm d}r,  \qquad &&Z_{0}=0.
 \end{aligned}
 \right.
\end{equation}
Here $\texttt{{\rm{x}}}\in\mathbb{D}$ and $\tau_{\texttt{{\rm{x}}}}$
is the first exit time of the trajectory $\texttt{{\rm{X}}}_{r}^{\texttt{{\rm{x}}}}$
to the boundary of $\mathbb{D}$.
 We apply the jump-adapted Euler scheme to system {\rm(\ref{eq:SDEs2})}
 and get
 \begin{equation}\label{eq:approSDEs13}
 \left\{
 \begin{aligned}
  &\widetilde{\texttt{{\rm{X}}}}_{\tau_{i+1}}=\widetilde{\texttt{{\rm{X}}}}_{\tau_{i}}
  +b(\widetilde{\texttt{{\rm{X}}}}_{\tau_{i}})(\tau_{i+1}-\tau_{i})
  +L_{\tau_{i+1}-\tau_{i}}^{\varepsilon},
  \\[3pt]
  &\widetilde{Y}_{\tau_{i+1}}=\widetilde{Y}_{\tau_{i}}+c(\widetilde{\texttt{{\rm{X}}}}_{\tau_{i}})
  \widetilde{Y}_{\tau_{i}}(\tau_{i+1}-\tau_{i}),
  \\[3pt]
  &\widetilde{Z}_{\tau_{i+1}}=\widetilde{Z}_{\tau_{i}}+f(\widetilde{\texttt{{\rm{X}}}}_{\tau_{i}})
  \widetilde{Y}_{\tau_{i}}(\tau_{i+1}-\tau_{i}),
 \end{aligned}
 \right.
\end{equation}
 where $\tau_{i},\,(i=0,1,\ldots, n_{\tau_{\texttt{{\rm{x}}}}})$ are jump times
 of $L_{t}^{\varepsilon}$ in {\rm(\ref{eq:approLevy})}.
 Thus we can approximate {\rm(\ref{eq:stableSde})} by
 \begin{equation}
   \mathbf{E}\left[
   g(\widetilde{\texttt{{\rm{X}}}}_{\widetilde{\tau}_{\texttt{{\rm{x}}}}})
   \widetilde{Y}_{\widetilde{\tau}_{\texttt{{\rm{x}}}}}
   +\widetilde{Z}_{\widetilde{\tau}_{\texttt{{\rm{x}}}}}\right],
 \end{equation}
 where $\widetilde{\tau}_{\texttt{{\rm{x}}}}$ is the exit time of the
 trajectory $\widetilde{\texttt{{\rm{X}}}}_{r}^{\texttt{{\rm{x}}}}$
 to the boundary of $\mathbb{D}$.

 If $b_{i}(\texttt{{\rm{x}}}),\,(i=1,2,\ldots,n)$, $f(\texttt{{\rm{x}}})$, and
 $c(\texttt{{\rm{x}}})$, together with their first derivatives with respect to
 $\texttt{{\rm{x}}}$ being uniformly bounded, then we have
 \begin{equation}
   \left|\mathbf{E}\left[g(\widetilde{\texttt{{\rm{X}}}}_{\widetilde{\tau}_{\texttt{{\rm{x}}}}})
   \widetilde{Y}_{\widetilde{\tau}_{\texttt{{\rm{x}}}}}
   +\widetilde{Z}_{\widetilde{\tau}_{\texttt{{\rm{x}}}}}\right]-u(\texttt{{\rm{x}}})\right|
   \leq C_{1}(\varepsilon^{2s}+\varepsilon^{\beta-2s}),
 \end{equation}
 where $C_{1}$ is a constant.
\end{remark}

\section{High-order scheme and error estimate}
In this section, we will give a higher-order jump-adapted scheme under Assumption
I with $\beta\in[2,3]$ and Assumption II. By the idea of \cite{Asmussen&Rosinski2001},
we approximate the symmetric $2s$-stable process
$L_{t}$ in (\ref{levyprocess}) by
\begin{equation}
 \widetilde{L}_{t}^{\varepsilon}=\sigma_{\varepsilon}W_{t}
 +\int_{0}^{t}\int_{\varepsilon<|\texttt{{\rm{y}}}|<1}\texttt{{\rm{y}}}
 \widetilde{N}({\rm d}\eta,{\rm d}\texttt{{\rm{y}}})+\int_{0}^{t}\int_{|\texttt{{\rm{y}}}|\geq1}
 \texttt{{\rm{y}}} N({\rm d}\eta,{\rm d}\texttt{{\rm{y}}})
 =\sigma_{\varepsilon}W_{t}+L_{t}^{\varepsilon},
\end{equation}
where $W_{t}=(W_{t}^{1},W_{t}^{2},\ldots,W_{t}^{n})$ is an $n$-dimensional Brownian motion,
\begin{equation}
  \sigma_{\varepsilon}\sigma_{\varepsilon}^{T}=C(n,s)\left(\int_{|\texttt{{\rm{y}}}|<\varepsilon}
  y_{j_1}y_{j_2}\nu({\rm d}\texttt{{\rm{y}}})\right)_{1\leq j_{1},j_{2}\leq n},
\end{equation}
and
\begin{equation}
 \sigma_{\varepsilon}=\left[C(n,s)\frac{\varepsilon^{2-2s}}{2-2s}
 \frac{\pi^{n/2}}{\Gamma(n/2+1)}\right]^{1/2}\mathrm{I}:=\overline{\sigma}_{\varepsilon}\mathrm{I}.
\end{equation}
Here $\mathrm{I}$ is an identity matrix of dimension $N$.

We have the weak convergent order in the following lemma.
\begin{lemma}
 Let $s\in(0,1)$, $h\in C^{\beta}(\mathbb{R}^{n})$, $\beta\in[2,3]$. Then there is a constant
 $C_{1}$, such that for $0\leq t'\leq t\leq T$, we have
 \begin{equation}
  \left|\mathbf{E}\left[h(L_{t-t'})-h(\widetilde{L}_{t-t'}^{\varepsilon})
  \right]\right|
  \leq C_{1}\varepsilon^{[\beta]^{-}-2s}|h|_{\beta}(t-t').
 \end{equation}
\end{lemma}
\noindent The proof is similar to that of Lemma 6 in \cite{Mikulevicius2012},
so we omit it here.

Since the exact exit time of  stochastic process with Brownian motion is hard to approximate,
we replace $\sigma_{\varepsilon}W_{t}$ by $\sigma_{\varepsilon}\sqrt{t}\xi$,
where $\xi=(\xi_{1},\xi_{2},\cdots,\xi_{n})$ and $\mathrm{P}(\xi_{1}=\pm1)=\mathrm{P}(\xi_{2}=\pm1)
=\cdots=\mathrm{P}(\xi_{n}=\pm1)=\frac{1}{2}$. Consider the following high-order jump-adapted
time discretization:
\begin{equation}
 \left\{
 \begin{aligned}
 &\tau_{0}=t,
 \\
 &{\tau}_{i+1}=\inf\{t>{\tau_{i}}:\Delta{L_{t}^{\varepsilon}}
 :=L_{t}^{\varepsilon}-L_{t-}^{\varepsilon}\neq 0\}\wedge
 \overline{\tau}_{t,\texttt{{\rm{x}}}}\wedge{T},
 \end{aligned}
 \right.
\end{equation}
where $\overline{\tau}_{t,\texttt{{\rm{x}}}}=\left\{r\geq{t},
\overline{\texttt{{\rm{X}}}}_{\eta}\notin{\mathbb{D}}\right\}$ and
$\overline{\texttt{{\rm{X}}}}_{\eta}$ is given by the following system.
For $r\in[\tau_{i},\tau_{i+1})$, define
\begin{equation}\label{eq:approSDEs2}
 \left\{
 \begin{aligned}
  &\overline{\texttt{{\rm{X}}}}_{r}
  =\overline{\texttt{{\rm{X}}}}_{\tau_{i}}
  +b({\tau_{i}},\overline{\texttt{{\rm{X}}}}_{{\tau}_{i}})(r-{\tau}_{i})
  +\sigma_{\varepsilon}\sqrt{r-{\tau}_{i}}\,\xi,
  \\[3pt]
  &\overline{Y}_{r}=\overline{Y}_{{\tau}_{i}}
  +c({\tau}_{i},\overline{\texttt{{\rm{X}}}}_{{\tau}_{i}})
  \overline{Y}_{{\tau}_{i}}(r-{\tau}_{i}),
  \\[3pt]
  &\overline{Z}_{r}=\overline{Z}_{{\tau}_{i}}
  +f({\tau}_{i},\overline{\texttt{{\rm{X}}}}_{{\tau}_{i}})
  \overline{Y}_{{\tau}_{i}}(r-{\tau}_{i}),
 \end{aligned}
 \right.
\end{equation}
where the starting point $\overline{\texttt{{\rm{X}}}}_{t}=\texttt{{\rm{x}}}$,
$\overline{Y}_{t}=1$, and $\overline{Z}_{t}=0$. For $r=\tau_{i+1}$, we define
\begin{equation}\label{eq:approSDEs22}
 \left\{
 \begin{aligned}
  &\overline{\texttt{{\rm{X}}}}_{\tau_{i+1}}
  =\overline{\texttt{{\rm{X}}}}_{\tau_{i+1}-}
  +\Delta{L_{\tau_{i+1}}^{\varepsilon}}
  \\[3pt]
  &\overline{Y}_{{\tau}_{i+1}}=\overline{Y}_{{\tau}_{i}}
  +c({\tau}_{i},\overline{\texttt{{\rm{X}}}}_{{\tau}_{i}})
  \overline{Y}_{{\tau}_{i}}({\tau}_{i+1}-{\tau}_{i}),
  \\[3pt]
  &\overline{Z}_{{\tau}_{i+1}}=\overline{Z}_{{\tau}_{i}}
  +f({\tau}_{i},\overline{\texttt{{\rm{X}}}}_{{\tau}_{i}})
  \overline{Y}_{{\tau}_{i}}({\tau}_{i+1}-{\tau}_{i}).
 \end{aligned}
 \right.
\end{equation}
 Thus, we approximate the solution $u(t,\texttt{{\rm{x}}})$ of equation (\ref{eq:terminalboundaryproblem})
by
\begin{equation}\label{eq:approSolution2}
 \mathbf{E}\left[
   u(T\wedge\overline{\tau}_{t,\texttt{{\rm{x}}}},
   \overline{\texttt{{\rm{X}}}}_{T\wedge\overline{\tau}_{t,\texttt{{\rm{x}}}}})
   \overline{Y}_{T\wedge\overline{\tau}_{t,\texttt{{\rm{x}}}}}
   +\overline{Z}_{T\wedge\overline{\tau}_{t,\texttt{{\rm{x}}}}}\right].
\end{equation}
We summarize the method in Algorithm 2.
\begin{algorithm}[t]\label{algo:sde2}
\caption{Algorithm for (\ref{eq:approSolution2})} 
\begin{algorithmic}[1]
\State
{\bf initialize:}
$\tau_{0}=t$, $\overline{\texttt{{\rm{X}}}}_{\tau_{0}}=\texttt{{\rm{x}}}$,
$\overline{Y}_{\tau_{0}}=1$, and $\overline{Z}_{\tau_{0}}=0$, $i=0$.
\While{$\tau_{i}<T$ and $\overline{\texttt{{\rm{X}}}}_{\tau_{i}}\in D$} 
　　\State {\bf Sample:} jump time $\tau$ ($\tau$ is exponentially distributed with parameter
    $\lambda_{\varepsilon}$ in (\ref{eq:jump intensity}))
    and $\xi=(\xi_{1},\,\xi_{2},\,\ldots,\xi_{n})$ with
    $\mathrm{P}(\xi_{1}=\pm1)=\mathrm{P}(\xi_{2}=\pm1)=\ldots=\mathrm{P}(\xi_{n}=\pm1)=\frac{1}{2}$.
    \\
\hspace*{0.18in} {\bf Set:} $\overline{\texttt{{\rm{X}}}}_{\tau_{i+1}^{-}}=\overline{\texttt{{\rm{X}}}}_{\tau_{i}}
  +b(\tau_{i},\overline{\texttt{{\rm{X}}}}_{\tau_{i}})\tau+\sigma_{\varepsilon}\tau^{\frac{1}{2}}\xi$.
     \If{$\tau+\tau_{i}>T$ or $\overline{\texttt{{\rm{X}}}}_{\tau_{i+1}^{-}}\notin \overline{D}$},
　　　　 \State {\bf Set:} $\tau_{i+1}=\sup\left\{\tau_{i}<t<T:\overline{\texttt{{\rm{X}}}}_{\tau_{i}}
           +b(\tau_{i},\overline{\texttt{{\rm{X}}}}_{\tau_{i}})(t-\tau_{i})
           +\sigma_{\varepsilon}\sqrt{t-\tau_{i}}\in \overline{D}\right\}$.
       \State {\bf Evaluate:} $\overline{\texttt{{\rm{X}}}}_{\tau_{i+1}}$ without jump,
         $\overline{Y}_{\tau_{i+1}}$, $\overline{Z}_{\tau_{i+1}}$ according to (\ref{eq:approSDEs2}).
       \State{\bf Set:} $i=i+1$.
       \Else
　　　　\State {\bf Sample:} jump size $J_{\varepsilon}$ according to density
        (\ref{eq:PossionRandomMeasure}).
         \State {\bf Evaluate:} $\overline{\texttt{{\rm{X}}}}_{\tau_{i+1}}$,
         $\overline{Y}_{\tau_{i+1}}$, $\overline{Z}_{\tau_{i+1}}$ according to (\ref{eq:approSDEs2}).
         \State{\bf Set:} $\tau_{i+1}=\tau_{i}+\tau$ and $i=i+1$.
   \EndIf \State {\bf end if}
\EndWhile \State {\bf end while}
\If{$\tau_{i+1}<T$} {\bf Set:} $\overline\tau_{t,\texttt{{\rm{x}}}}=\tau_{i+1}$,
 {\bf Evaluate:} $\chi(\overline{\tau}_{t,\texttt{{\rm{x}}}},
   \overline{\texttt{{\rm{X}}}}_{\overline{\tau}_{t,\texttt{{\rm{x}}}}})
   \overline{Y}_{\overline{\tau}_{t,\texttt{{\rm{x}}}}}
   +\overline{Z}_{\overline{\tau}_{t,\texttt{{\rm{x}}}}}$.
    \Else \quad
    {\bf Set:} $\overline\tau_{t,\texttt{{\rm{x}}}}=T$,
 {\bf Evaluate:} $g(T,\overline{\texttt{{\rm{X}}}}_{T})\overline{Y}_{T}+\overline{Z}_{T}$.
\EndIf \State {\bf end if}
\State Loop above algorithm $N$ times.
\State{\bf Evaluate:}
$u(t,\texttt{{\rm{x}}})\approx \frac{1}{N}
\sum\limits_{j=1}^{N}\left[u(T\wedge\overline{\tau}_{t,\texttt{{\rm{x}}}},
   \overline{\texttt{{\rm{X}}}}_{T\wedge\overline{\tau}_{t,\texttt{{\rm{x}}}}}^{j})
   \overline{Y}_{T\wedge\overline{\tau}_{t,\texttt{{\rm{x}}}}}^{j}
   +\overline{Z}_{T\wedge\overline{\tau}_{t,\texttt{{\rm{x}}}}}^{j}\right]$.
\end{algorithmic}
\end{algorithm}

Before proving the convergence theorem for Algorithm 2, we give the following lemma
which is similar to Lemma \ref{lemma:Ybound}.
\begin{lemma}\label{lemma:Ybound2}
 Under Assumption {\rm III}, $\overline{Y}_{\tau_{i}}$ defined
  in {\rm (\ref{eq:approSDEs22})}
 is uniformly bounded by a deterministic formula
 \begin{equation}
  |\overline{Y}_{\tau_{i}}|<e^{\overline{c}(T-t)}, \quad i=0,1,\ldots,
  n_{T\wedge{\overline{\tau}_{t,\texttt{{\rm{x}}}}}},
 \end{equation}
 where $\overline{c}=\max\limits_{(t,\texttt{{\rm{x}}})\in[0,T]\times \mathbb{D}}c(t,\texttt{{\rm{x}}})$.
\end{lemma}
\begin{theorem}\label{thm:Errorestimate2}
 Let $\varepsilon\in(0,1)$. Under Assumption {\rm I} with $\beta\in[2,3]$ and
 {\rm III}, the error estimate
 \begin{equation}
  \left|\mathbf{E}\left[
   u(T\wedge\overline{\tau}_{t,\texttt{{\rm{x}}}},
   \overline{\texttt{{\rm{X}}}}_{T\wedge\overline{\tau}_{t,\texttt{{\rm{x}}}}})
   \overline{Y}_{T\wedge\overline{\tau}_{t,\texttt{{\rm{x}}}}}
   +\overline{Z}_{T\wedge\overline{\tau}_{t,\texttt{{\rm{x}}}}}\right]
    -u(t,\texttt{{\rm{x}}})\right|
    \leq C_{1}\varepsilon^{2s}+C_{2}\varepsilon^{\lfloor\beta\rfloor-2s}
 \end{equation}
 holds with $C_{1}$ and $C_{2}$ being constants independent of $t$ and $\texttt{{\rm{x}}}$.
\end{theorem}
\begin{proof}
 We only prove the case of $\beta=3$.
  Other cases can be proved similarly.
 \begin{equation}\label{eq:errorestimate_high}
  \begin{aligned}
  &\big|\mathbf{E}\left[
   u(T\wedge\overline{\tau}_{t,\texttt{{\rm{x}}}},
   \overline{\texttt{{\rm{X}}}}_{T\wedge\overline{\tau}_{t,\texttt{{\rm{x}}}}})
   \overline{Y}_{T\wedge\overline{\tau}_{t,\texttt{{\rm{x}}}}}
   +\overline{Z}_{T\wedge\overline{\tau}_{t,\texttt{{\rm{x}}}}}\right]
    -u(t,\texttt{{\rm{x}}})\big|
    \\[3pt]
   =& \left|\mathbf{E}\left\{\sum_{i=0}^{n_{T\wedge\overline{\tau}_{t,\texttt{{\rm{x}}}}}-1}
   \Big[u({\tau}_{i+1},\overline{\texttt{{\rm{X}}}}_{{\tau}_{i+1}})
   \overline{Y}_{{\tau}_{i+1}}+\overline{Z}_{{\tau}_{i+1}}
   -u({\tau}_{i},\overline{\texttt{{\rm{X}}}}_{{\tau}_{i}})
   \overline{Y}_{\tau_{i}}-\overline{Z}_{{\tau}_{i}}\Big]\right\}\right|
   \\[3pt]
   =& \Bigg|\mathbf{E}\Bigg\{\sum_{i=0}^{n_{T\wedge\overline{\tau}_{t,\texttt{{\rm{x}}}}}-1}
    \overline{Y}_{\tau_{i+1}}\left[u({\tau}_{i+1},\overline{\texttt{{\rm{X}}}}_{{\tau}_{i+1}})
   -u({\tau}_{i+1},\widetilde{\texttt{{\rm{X}}}}_{{\tau}_{i+1}})\right]
   \\[3pt]
   &+\overline{Y}_{\tau_{i}}\left[u({\tau}_{i+1},\widetilde{\texttt{{\rm{X}}}}_{{\tau}_{i+1}^{-}})
   -u({\tau}_{i},\overline{\texttt{{\rm{X}}}}_{{\tau}_{i}})\right]
   +\overline{Y}_{\tau_{i}}c({\tau}_{i},\overline{\texttt{{\rm{X}}}}_{{\tau}_{i}})
   u({\tau}_{i+1},\widetilde{\texttt{{\rm{X}}}}_{{\tau}_{i+1}^{-}})\Delta{\tau_{i}}
   \\[3pt]
   &+\overline{Y}_{\tau_{i+1}}\left[u({\tau}_{i+1},\widetilde{\texttt{{\rm{X}}}}_{{\tau}_{i+1}})
   -u({\tau}_{i+1},\widetilde{\texttt{{\rm{X}}}}_{{\tau}_{i+1}^{-}})\right]
   +\overline{Y}_{\tau_{i}}f({\tau}_{i},\overline{\texttt{{\rm{X}}}}_{{\tau}_{i}})\Delta{\tau_{i}}
   \Bigg\}\Bigg|,
  \end{aligned}
 \end{equation}
 where
 \begin{equation}\label{eq:approSDEs3}
  \left\{
  \begin{aligned}
  &\widetilde{\texttt{{\rm{X}}}}_{\eta}
  =\overline{\texttt{{\rm{X}}}}_{\tau_{i}}+
  b({\tau_{i}},\overline{{\texttt{{\rm{X}}}}}_{{\tau}_{i}})(\eta-\tau_{i})
  +\sigma_{\varepsilon}W_{\eta-\tau_{i}},\quad  \eta\in [\tau_{i},\tau_{i+1}),
  \\[3pt]
  &\widetilde{\texttt{{\rm{X}}}}_{\tau_{i+1}}=\widetilde{\texttt{{\rm{X}}}}_{\tau_{i+1}-}
  +\Delta L_{\tau_{i+1}}^{\varepsilon},
  \end{aligned}
  \right.
 \end{equation}
and $\Delta{\tau_{i}}=\tau_{i+1}-\tau_{i}$.

 By It\^{o} formula, the martingale property and system (\ref{eq:approSDEs3}), one gets
\begin{align}\label{eq:errorestimate_high1}
  &\big|\mathbf{E}\left[
   u(T\wedge\overline{\tau}_{t,\texttt{{\rm{x}}}},
   \overline{\texttt{{\rm{X}}}}_{T\wedge\overline{\tau}_{t,\texttt{{\rm{x}}}}})
   \overline{Y}_{T\wedge\overline{\tau}_{t,\texttt{{\rm{x}}}}}
   +\overline{Z}_{T\wedge\overline{\tau}_{t,\texttt{{\rm{x}}}}}\right]
    -u(t,\texttt{{\rm{x}}})\big|
    \nonumber\\[3pt]
   =&\Bigg|\mathbf{E}\Bigg\{\sum_{i=0}^{n_{T\wedge\overline{\tau}_{t,\texttt{{\rm{x}}}}}-1}
   \overline{Y}_{\tau_{i}}\int_{\tau_{i}}^{\tau_{i+1}}\bigg[\frac{\partial{u}}{\partial{\eta}}
   (\eta,\widetilde{\texttt{{\rm{X}}}}_{\eta-})
   +\sum_{j=1}^{n}b_{j}(\tau_{i},\widetilde{\texttt{{\rm{X}}}}_{\tau_{i}})
   \frac{\partial{u}}{\partial{x}_{j}}(\eta,\widetilde{\texttt{{\rm{X}}}}_{\eta-})
   \nonumber\\[3pt]
   &+\frac{1}{2}\sum_{j_{1},j_{2}=1}^{n}(\sigma_{\varepsilon}^{T}
   \sigma_{\varepsilon})_{j_{1},j_{2}}
   \frac{\partial^{2}u(\eta,\widetilde{\texttt{{\rm{X}}}}_{\eta-})}
   {\partial{x_{j_{1}}}\partial{x_{j_{2}}}}\bigg]{\rm d}\eta
   +\overline{Y}_{\tau_{i}}c({\tau}_{i},\overline{\texttt{{\rm{X}}}}_{{\tau}_{i}})
   u({\tau}_{i+1},\widetilde{\texttt{{\rm{X}}}}_{{\tau}_{i+1}^{-}})\Delta{\tau_{i}}
   \nonumber\\[3pt]
   &+\overline{Y}_{\tau_{i+1}}\left[u({\tau}_{i+1},\overline{\texttt{{\rm{X}}}}_{{\tau}_{i+1}})
   -u({\tau}_{i+1},\widetilde{\texttt{{\rm{X}}}}_{{\tau}_{i+1}})\right]
   +\overline{Y}_{\tau_{i}}f({\tau}_{i},\overline{\texttt{{\rm{X}}}}_{{\tau}_{i}})\Delta{\tau_{i}}
   \Bigg\}
   \nonumber\\[3pt]
   &+\mathbf{E}\left[\int_{t}^{T\wedge\widetilde{\tau}_{t,\texttt{{\rm{x}}}}}
   \widetilde{Y}_{\eta}\int_{|\texttt{{\rm{y}}}|>\varepsilon}\left[
   u(\eta,\widetilde{\texttt{{\rm{X}}}}_{\eta-}+\texttt{{\rm{y}}})
   -u(\eta,\widetilde{\texttt{{\rm{X}}}}_{\eta-})\right]N({\rm d}\eta,{\rm d}\texttt{{\rm{y}}})
   \right]\Bigg|
    \nonumber\\[3pt]
   =&\Bigg|\mathbf{E}\Bigg\{\sum_{i=0}^{n_{T\wedge\overline{\tau}_{t,\texttt{{\rm{x}}}}}-1}
   \overline{Y}_{\tau_{i}}\mathbf{E}\bigg[\int_{\tau_{i}}^{\tau_{i+1}}
   \bigg[\frac{\partial{u}}{\partial{\eta}}
   (\eta,\widetilde{\texttt{{\rm{X}}}}_{\eta-})
   +\sum_{j=1}^{n}b_{j}(\tau_{i},\widetilde{\texttt{{\rm{X}}}}_{\tau_{i}})
   \frac{\partial{u}}{\partial{x}_{j}}(\eta,\widetilde{\texttt{{\rm{X}}}}_{\eta-})
   \nonumber\\[3pt]
   &+\frac{1}{2}\sum_{j_{1},j_{2}=1}^{n}(\sigma_{\varepsilon}
   \sigma_{\varepsilon}^{T})_{j_{1},j_{2}}
   \frac{\partial^{2}u(\eta,\widetilde{\texttt{{\rm{X}}}}_{\eta-})}
   {\partial{x_{j_{1}}}\partial{x_{j_{2}}}}\bigg]{\rm d}\eta
   +c({\tau}_{i},\overline{\texttt{{\rm{X}}}}_{{\tau}_{i}})
   u({\tau}_{i+1},\widetilde{\texttt{{\rm{X}}}}_{{\tau}_{i+1}^{-}})\Delta{\tau_{i}}
   \nonumber\\[3pt]
   &+\left[1+c({\tau}_{i},\overline{\texttt{{\rm{X}}}}_{{\tau}_{i}})\Delta{\tau_{i}}\right]
   \left[u({\tau}_{i+1},\overline{\texttt{{\rm{X}}}}_{{\tau}_{i+1}})
   -u({\tau}_{i+1},\widetilde{\texttt{{\rm{X}}}}_{{\tau}_{i+1}})\right]
   +f({\tau}_{i},\overline{\texttt{{\rm{X}}}}_{{\tau}_{i}})\Delta{\tau_{i}}
   \nonumber\\[3pt]
   &+\left(1+c({\tau}_{i},\overline{\texttt{{\rm{X}}}}_{{\tau}_{i}})\Delta{\tau_{i}}\right)
   \int_{\tau_{i}}^{\tau_{i+1}}
   \int_{|\texttt{{\rm{y}}}|>\varepsilon}\left[
   u(\eta,\widetilde{\texttt{{\rm{X}}}}_{\eta-}+\texttt{{\rm{y}}})
   -u(\eta,\widetilde{\texttt{{\rm{X}}}}_{\eta-})\right]\nu({\rm d}\texttt{{\rm{y}}}){\rm d}\eta
   \big{|}\mathcal{F}_{\tau_{i}}\bigg]\Bigg\}\Bigg|
   \nonumber\\[3pt]
   =& \left|\mathbf{E}\left\{\sum_{i}\overline{Y}_{\tau_{i}}\mathbf{E}\left[\,
   \overline{B}_{i}\big{|}\mathcal{F}_{{\tau}_{i}}\right]\right\}\right|.
 \end{align}
 Since $u(t,\texttt{{\rm{x}}})$ satisfies (\ref{eq:terminalboundaryproblem}), it holds that
 \begin{align}\label{eq:errorestimate2}
   &\mathbf{E}\left[\overline{B}_{i}\big{|}\mathcal{F}_{{\tau}_{i}}\right]
   \nonumber\\[3pt]
   =&\mathbf{E}\Bigg\{
   \int_{\tau_{i}}^{\tau_{i+1}}\bigg[\frac{1}{2}\sum_{j_{1},j_{2}=1}^{n}(\sigma_{\varepsilon}
   \sigma_{\varepsilon}^{T})_{j_{1},j_{2}}
   \frac{\partial^{2}u}{\partial{x_{j_{1}}}\partial{x_{j_{2}}}}
   (\eta,\widetilde{\texttt{{\rm{X}}}}_{\eta-})
   \nonumber\\[3pt]
   &-\int_{|\texttt{{\rm{y}}}|<\varepsilon}\big[
   u(\eta,\widetilde{\texttt{{\rm{X}}}}_{\eta-}+\texttt{{\rm{y}}})
   -u(\eta,\widetilde{\texttt{{\rm{X}}}}_{\eta-})
   -I_{|\texttt{{\rm{y}}}|<1}
   (\nabla{u(\eta,\widetilde{\texttt{{\rm{X}}}}_{\eta-})},y)\big]\nu({\rm d}\texttt{{\rm{y}}})
   \bigg]\,{\rm d}\eta
   \nonumber\\[3pt]
   &+c(\tau_{i},\widetilde{\texttt{{\rm{X}}}}_{\tau_{i}})\Delta{\tau_{i}}
   \int_{\tau_{i}}^{\tau_{i+1}}\int_{|\texttt{{\rm{y}}}|>\varepsilon}\left[
   u(\eta,\widetilde{\texttt{{\rm{X}}}}_{\eta-}+\texttt{{\rm{y}}})
   -u(\eta,\widetilde{\texttt{{\rm{X}}}}_{\eta-})\right]
   \nu({\rm d}\texttt{{\rm{y}}}){\rm d}\eta
   \nonumber\\[3pt]
   &+\left[1+c({\tau}_{i},\overline{\texttt{{\rm{X}}}}_{{\tau}_{i}})\Delta{\tau_{i}}\right]
   \left[u\left({\tau}_{i+1},\overline{\texttt{{\rm{X}}}}_{{\tau}_{i+1}}\right)
   -u\left({\tau}_{i+1},\widetilde{\texttt{{\rm{X}}}}_{{\tau}_{i+1}}\right)\right]
   \nonumber\\[3pt]
   &+\int_{\tau_{i}}^{\tau_{i+1}}\sum_{j=1}^{n}\left[b_{j}
   (\tau_{i},\overline{\texttt{{\rm{X}}}}_{\tau_{i}})
   -b_{j}(\eta,\widetilde{\texttt{{\rm{X}}}}_{\eta-})\right]
   \frac{\partial{u}}{\partial{x}_{j}}(\eta,\widetilde{\texttt{{\rm{X}}}}_{\eta-}){\rm d}\eta
   \nonumber\\[3pt]
   &+\int_{\tau_{i}}^{\tau_{i+1}}c(\tau_{i},\overline{\texttt{{\rm{X}}}}_{\tau_{i}})
   u(\tau_{i+1},\overline{\texttt{{\rm{X}}}}_{\tau_{i+1}^{-}})
   -c(\eta,\widetilde{\texttt{{\rm{X}}}}_{\eta-})u(\eta,\widetilde{\texttt{{\rm{X}}}}_{\eta-})
   {\rm d}\eta
   \nonumber\\[3pt]
   &+\int_{\tau_{i}}^{\tau_{i+1}}f(\tau_{i},\overline{\texttt{{\rm{X}}}}_{\tau_{i}})
   -f(\eta,\widetilde{\texttt{{\rm{X}}}}_{\eta-}){\rm d}\eta
   \Big{|}\mathcal{F}_{\tau_{i}}\Bigg\}
   \nonumber\\[3pt]
   =&\,\mathbf{E}\left[\overline{B}_{i1}+\overline{B}_{i2}
   +\overline{B}_{i3}+\overline{B}_{i4}+\overline{B}_{i5}+\overline{B}_{i6}
   |\mathcal{F}_{\tau_{i}}\right],
  \end{align}

 For the first term, we obtain
 \begin{align}\label{eq:estimateBi1}
   &\left|\mathbf{E}\left[\overline{B}_{i1}\Big{|}\mathcal{F}_{\tau_{i}}\right]\right|
   \nonumber\\[3pt]
   =&\,C(n,s)\bigg|\mathbf{E}\bigg[\int_{\tau_{i}}^{\tau_{i+1}}
   \frac{1}{2}\sum_{j_{1},j_{2}=1}^{n}\int_{|\texttt{{\rm{y}}}|<\varepsilon}
   y_{j_{1}}y_{j_{2}}\frac{\partial^{2}u}{\partial{x_{j_{1}}}\partial{x_{j_{2}}}}
   (\eta,\widetilde{\texttt{{\rm{X}}}}_{\eta-})\frac{{\rm d}\texttt{{\rm{y}}}}
   {|\texttt{{\rm{y}}}|^{n+2s}}
   \nonumber\\[3pt]
   &-\int_{|\texttt{{\rm{y}}}|<\varepsilon}\int_{0}^{1}
   (\nabla{u(\eta,\widetilde{\texttt{{\rm{X}}}}_{\eta-}+\alpha\texttt{{\rm{y}}})}
   -\nabla{u(\eta,\widetilde{\texttt{{\rm{X}}}}_{\eta-})},\texttt{{\rm{y}}})
   {\rm d}\alpha\,\frac{{\rm d}\texttt{{\rm{y}}}}
   {|\texttt{{\rm{y}}}|^{n+2s}}\,{\rm d}\eta\Big{|}\mathcal{F}_{\tau_{i}}\bigg]\bigg|
   \nonumber\\[3pt]
   =&\,C(n,s)\bigg|\mathbf{E}\bigg[\int_{\tau_{i}}^{\tau_{i+1}}
   \sum_{j_{1},j_{2}=1}^{n}\int_{|\texttt{{\rm{y}}}|<\varepsilon}
   \int_{0}^{1}\int_{0}^{1}\alpha y_{j1}y_{j2}
   \nonumber\\[3pt]
   &\times\Big[\frac{\partial^{2}u}
   {\partial{x_{j_{1}}}\partial{x_{j_{2}}}}(\eta,\widetilde{\texttt{{\rm{X}}}}_{\eta-}
   +\alpha'\alpha\texttt{{\rm{y}}})-\frac{\partial^{2}u}{\partial{x_{j_{1}}}\partial{x_{j_{2}}}}
   (\eta,\widetilde{\texttt{{\rm{X}}}}_{\eta-})\Big]
   {\rm d}\alpha'{\rm d}\alpha \frac{{\rm d}\texttt{{\rm{y}}}}
   {|\texttt{{\rm{y}}}|^{n+2s}} {\rm d}\eta\Big{|}\mathcal{F}_{\tau_{i}}\bigg]\bigg|
   \nonumber\\[3pt]
   \leq&\,C_{2}\mathbf{E}\bigg[\int_{\tau_{i}}^{\tau_{i+1}}\int_{|\texttt{{\rm{y}}}|<\varepsilon}
   \int_{0}^{1}\int_{0}^{1}\left|D^2u(\eta,\widetilde{\texttt{{\rm{X}}}}_{\eta-}
   +\alpha'\alpha\texttt{{\rm{y}}})
   -D^2u(\eta,\widetilde{\texttt{{\rm{X}}}}_{\eta-})\right|
   \frac{{\rm d}\alpha'{\rm d}\alpha{\rm d}\texttt{{\rm{y}}}{\rm d}\eta}
   {|\texttt{{\rm{y}}}|^{n+2s-2}}
   \Big{|}\mathcal{F}_{\tau_{i}}\bigg]\bigg|
   \nonumber\\[3pt]
   \leq&\, C_{2}\varepsilon^{3-2s}\mathbf{E}\left[\tau_{i+1}-\tau_{i}
   \Big{|}\mathcal{F}_{\tau_{i}}\right].
 \end{align}

  By Taylor expansion, it follows for the third part that
  \begin{align}\label{eq:estimateBi2}
    &\left|\mathbf{E}\left[\overline{B}_{i3}\Big{|}\mathcal{F}_{\tau_{i}}\right]\right|
    \nonumber\\[3pt]
    =&\left|
    \mathbf{E}\left\{\left[1+c({\tau}_{i},\overline{\texttt{{\rm{X}}}}_{{\tau}_{i}})
    \Delta{\tau_{i}}\right]\left[u\left({\tau}_{i+1},\overline{\texttt{{\rm{X}}}}_{{\tau}_{i+1}}\right)
    -u\left({\tau}_{i+1},\widetilde{\texttt{{\rm{X}}}}_{{\tau}_{i+1}}\right)\right]
   \Big{|}\mathcal{F}_{\tau_{i}}\right\}\right|
   \nonumber\\[3pt]
    =&\bigg|\mathbf{E}\bigg\{\left[1+c({\tau}_{i},\overline{\texttt{{\rm{X}}}}_{{\tau}_{i}})
    \Delta{\tau_{i}}\right]
    \nonumber\\[3pt]
    &\times \left[u\left({\tau}_{i+1},\overline{\texttt{{\rm{X}}}}_{{\tau}_{i}}
    +L_{\Delta\tau_{i}}^{\varepsilon}+\overline{\sigma}_{\varepsilon}(\Delta\tau_{i})^{\frac{1}{2}}
    \xi\right)-u\left({\tau}_{i+1},\overline{\texttt{{\rm{X}}}}_{{\tau}_{i}}
    +L_{\Delta\tau_{i}}^{\varepsilon}+\overline{\sigma}_{\varepsilon}W_{\Delta\tau_{i}}\right)\right]
    \Big{|}\mathcal{F}_{\tau_{i}}\bigg]\bigg|
    \nonumber\\[3pt]
    =&\,\bigg|\mathbf{E}\bigg\{\left[1+c({\tau}_{i},\overline{\texttt{{\rm{X}}}}_{{\tau}_{i}})
    \Delta{\tau_{i}}\right]\Big[\sum_{j_{1}=1}^{n}\frac{\partial{u}}{\partial{x_{j_{1}}}}
    \left(\tau_{i+1},\overline{\texttt{{\rm{X}}}}_{{\tau}_{i}}+L_{\Delta\tau_{i}}^{\varepsilon}\right)
    \left[(\Delta\tau_{i})^{\frac{1}{2}}\overline{\sigma}_{\varepsilon}\xi_{j_{1}}
    -\overline{\sigma}_{\varepsilon}W_{\Delta\tau_{i}}^{j_{1}}\right]
    \nonumber\\[3pt]
    &+\frac{1}{2!}\sum_{j_{1},j_{2}=1}^{n}\frac{\partial^2{u}}{\partial{x_{j_{1}}}
    \partial{x_{j_{2}}}}\left(\tau_{i+1},\overline{\texttt{{\rm{X}}}}_{{\tau}_{i}}
    +L_{\Delta\tau_{i}}^{\varepsilon}\right)
    \Big[\Delta\tau_{i}\overline{\sigma}_{\varepsilon}\xi_{j_{1}}\,\overline{\sigma}_{\varepsilon}\xi_{j_{2}}
    -\overline{\sigma}_{\varepsilon}W_{\Delta\tau_{i}}^{j_{1}}
    \overline{\sigma}_{\varepsilon}W_{\Delta\tau_{i}}^{j_{2}}\Big]
    \nonumber\\[3pt]
    &+\frac{1}{3!}\sum_{j_{1},j_{2},j_{3}=1}^{n}\frac{\partial^3{u}(\tau_{i+1},
    \overline{\texttt{{\rm{X}}}}_{{\tau}_{i}}
    +L_{\Delta\tau_{i}}^{\varepsilon}
    +\theta_{1}(\Delta\tau_{i})^{\frac{1}{2}}\xi)}
    {\partial{x_{j_{1}}}\partial{x_{j_{2}}}
    \partial{x_{j_{3}}}}\left((\Delta\tau_{i})^{\frac{3}{2}}
    \overline{\sigma}_{\varepsilon}\xi_{j_{1}}\,\overline{\sigma}_{\varepsilon}\xi_{j_{2}}\,
    \overline{\sigma}_{\varepsilon}\xi_{j_{3}}\right)
    \nonumber\\[3pt]
    &-\frac{\partial^3{u}(\tau_{i+1},\overline{\texttt{{\rm{X}}}}_{{\tau}_{i}}
    +L_{\Delta\tau_{i}}^{\varepsilon}
    +\theta_{2}\sigma_{\varepsilon}W_{\Delta\tau_{i}})
    }{\partial{x_{j_{1}}}\partial{x_{j_{2}}}
    \partial{x_{j_{3}}}}
    \left(\overline{\sigma}_{\varepsilon}W_{\Delta\tau_{i}}^{j_{1}}\,
    \overline{\sigma}_{\varepsilon}W_{\Delta\tau_{i}}^{j_{2}}\,
    \overline{\sigma}_{\varepsilon}W_{\Delta\tau_{i}}^{j_{3}}\right)\Big]
    \Big{|}\mathcal{F}_{\tau_{i}}\bigg\}\bigg|
    \nonumber\\[3pt]
    \leq& \,C_{2}\overline{\sigma}_{\varepsilon}^{3}
    \mathbf{E}\left[(\tau_{i+1}-\tau_{i})^{\frac{3}{2}}\big{|}\mathcal{F}_{\tau_{i}}\right]
    \leq C_{2}{\varepsilon}^{3-2s}\mathbf{E}\left[\tau_{i+1}-\tau_{i}\big{|}\mathcal{F}_{\tau_{i}}\right],
   \end{align}
  in which we utilize the fact that $\xi$, $W_{1}$ and $\tau_{i}\,(i=1,2,\ldots,n)$ are independent,
  \begin{equation}
   \begin{aligned}
   &\mathbf{E}\left[W_{\Delta\tau_{i}}\right]
   =\mathbf{E}[(\Delta\tau_{i})^\frac{1}{2}W_{1}], \quad
   \mathbf{E}\left[\overline{\sigma}_{\varepsilon}\xi_{j_{1}}\right]
   = \mathbf{E}\left[\overline{\sigma}_{\varepsilon}W_{1}^{j_{1}}\right]=0,
   \, j_{1}=1,2,\ldots,n,
   \\[3pt]
   &\mathbf{E}\left[\overline{\sigma}_{\varepsilon}\xi_{j_{1}}\,
   \overline{\sigma}_{\varepsilon}\xi_{j_{2}}\right]
   =\mathbf{E}\left[\overline{\sigma}_{\varepsilon}W_{1}^{j_{1}}\,
   \overline{\sigma}_{\varepsilon}W_{1}^{j_{2}}\right]
   =\left\{\begin{aligned}
   0, \quad j_{1}\neq j_{2},
   \\
   \overline{\sigma}_{\varepsilon}^2, \quad j_{1}=j_{2},
   \end{aligned}
   \right.
   \\
   \mathrm{and}
   \\[3pt]
   &\mathbf{E}\left[\left|\overline{\sigma}_{\varepsilon}\xi_{j_{1}}\,
   \overline{\sigma}_{\varepsilon}\xi_{j_{2}}\,
   \overline{\sigma}_{\varepsilon}\xi_{j_{3}}\right|\right]\leq
   \overline{\sigma}_{\varepsilon}^3,
   \quad
   \mathbf{E}\left[\left|\overline{\sigma}_{\varepsilon}W_{1}^{j_{1}}\,
   \overline{\sigma}_{\varepsilon}W_{1}^{j_{2}}
   \overline{\sigma}_{\varepsilon}W_{1}^{j_{3}}\right|\right]\leq
   \sqrt{15}\,\overline{\sigma}_{\varepsilon}^3.
   \end{aligned}
  \end{equation}
  Now we estimate the fourth term
  \begin{equation}
   \begin{aligned}
   &\left|\mathbf{E}\left[\overline{B}_{i4}\Big{|}\mathcal{F}_{\tau_{i}}\right]\right|
   =\left|\mathbf{E}\left\{\int_{\tau_{i}}^{\tau_{i+1}}\sum_{j=1}^{n}\left[b_{j}
   (\tau_{i},\overline{\texttt{{\rm{X}}}}_{\tau_{i}})
   -b_{j}(\eta,\widetilde{\texttt{{\rm{X}}}}_{\eta-})\right]
   \frac{\partial{u}}{\partial{x}_{j}}(\eta,\widetilde{\texttt{{\rm{X}}}}_{\eta-}){\rm d}\eta
   \Big{|}\mathcal{F}_{\tau_{i}}\right\}\right|
   \\[3pt]
   =&\Bigg|\mathbf{E}\left\{\int_{\tau_{i}}^{\tau_{i+1}}\sum_{j=1}^{n}\left[b_{j}
   (\tau_{i},\overline{\texttt{{\rm{X}}}}_{\tau_{i}})
   -b_{j}(\eta,\widetilde{\texttt{{\rm{X}}}}_{\eta-})\right]
   \frac{\partial{u}}{\partial{x}_{j}}(\eta,\widetilde{\texttt{{\rm{X}}}}_{\tau_{i}}){\rm d}\eta
   \Big{|}\mathcal{F}_{\tau_{i}}\right\}
   \\[3pt]
   &+\mathbf{E}\left\{\int_{\tau_{i}}^{\tau_{i+1}}\sum_{j=1}^{n}\left[b_{j}
   (\tau_{i},\overline{\texttt{{\rm{X}}}}_{\tau_{i}})
   -b_{j}(\eta,\widetilde{\texttt{{\rm{X}}}}_{\eta-})\right]
   \left[\frac{\partial{u}}{\partial{x}_{j}}(\eta,\widetilde{\texttt{{\rm{X}}}}_{\eta-})
   -\frac{\partial{u}}{\partial{x}_{j}}
   (\eta,\widetilde{\texttt{{\rm{X}}}}_{\tau_{i}})\right]{\rm d}\eta
   \Big{|}\mathcal{F}_{\tau_{i}}\right\}\Bigg|
   \\[3pt]
   =&\left|\mathbf{E}\left[\overline{B}_{i41}
   +\overline{B}_{i42}\Big{|}\mathcal{F}_{\tau_{i}}\right]\right|.
   \end{aligned}
  \end{equation}
By It\^{o} formula and the martingale property again, one has
 \begin{equation}
  \begin{aligned}
   &\left|\mathbf{E}\left[\overline{B}_{i41}\Big{|}\mathcal{F}_{\tau_{i}}\right]\right|
   \\[3pt]
   =&\Bigg|\mathbf{E}\bigg\{\sum_{j=1}^{n}
   \int_{\tau_{i}}^{\tau_{i+1}}\frac{\partial{u}}{{\partial{x}_{j}}}
   (\eta,\widetilde{\texttt{{\rm{X}}}}_{\tau_{i}})
   \int_{\tau_{i}}^{\eta}
   \bigg[\frac{\partial{b_{j}}}{\partial{\eta'}}
   (\eta',\widetilde{\texttt{{\rm{X}}}}_{\eta'-})
   +\sum_{j_{1}=1}^{n}b_{j_{1}}(\tau_{i},\widetilde{\texttt{{\rm{X}}}}_{\tau_{i}})
   \frac{\partial{b_{j}}}{\partial{x}_{j_{1}}}(\eta',\widetilde{\texttt{{\rm{X}}}}_{\eta'-})
   \\[3pt]
   &+\frac{1}{2}\sum_{j_{1},j_{2}=1}^{n}(\sigma_{\varepsilon}^{T}
   \sigma_{\varepsilon})_{j_{1},j_{2}}
   \frac{\partial^{2}u(\eta',\widetilde{\texttt{{\rm{X}}}}_{\eta'-})}
   {\partial{x_{j_{1}}}\partial{x_{j_{2}}}}\bigg]{\rm d}\eta'\,{\rm d}\eta
   \Big{|}\mathcal{F}_{\tau_{i}}\bigg\}
   \\[3pt]
   &+\mathbf{E}\left[\sum_{j=1}^{n}\int_{\tau_{i}}^{\tau_{i+1}}
   \frac{\partial{u}}{{\partial{x}_{j}}}
   (\eta,\widetilde{\texttt{{\rm{X}}}}_{\tau_{i}})
   \int_{\tau_{i}}^{\eta}\sum_{j_{1}=1}^{n}\overline{\sigma}_{\varepsilon}^{j_{1}}
   \frac{\partial{b_{j}}}{\partial{x}_{j_{1}}}
   (\eta',\widetilde{\texttt{{\rm{X}}}}_{\eta'-}){\rm d}W_{\eta'-}^{j_{1}}\,{\rm d}\eta
   \Big{|}\mathcal{F}_{\tau_{i}}\right]\Bigg|
   \\[3pt]
   \leq&C_{1}\mathbf{E}\left[(\tau_{i+1}-\tau_{i})^2\Big{|}\mathcal{F}_{\tau_{i}}\right],
  \end{aligned}
 \end{equation}
where $C_{1}$ is a constant. It is also easy to get
\begin{equation}\label{eq:estimateBi42}
   \mathbf{E}\left[\overline{B}_{i42}\Big{|}\mathcal{F}_{\tau_{i}}\right]
   \leq C_{1}\mathbf{E}\left[(\tau_{i+1}-\tau_{i})^2\Big{|}\mathcal{F}_{\tau_{i}}\right],
  \end{equation}
where the conditions of $u(t,\texttt{{\rm{x}}})$ and $b_{j}(t,\texttt{{\rm{x}}})$
are uesd.

For the remaining part, by Taylor expansion and It\^{o} formula, it is evident that
  \begin{equation}\label{eq:estimateBi3Bi4}
   \mathbf{E}\left[\overline{B}_{i2}+\overline{B}_{i5}+\overline{B}_{i6}
   \Big{|}\mathcal{F}_{\tau_{i}}\right]
   \leq C_{1}\mathbf{E}\left[(\tau_{i+1}-\tau_{i})^2\Big{|}\mathcal{F}_{\tau_{i}}\right].
  \end{equation}
  Here $C_{1}$ is a constant.

  Since Lemma \ref{lemma:jumptime} is still true by replacing
  $\delta_{i}=(\overline{\tau}_{t,\texttt{{\rm{x}}}}\wedge T)-\tau_{i}$.
  Thus, combining Lemmas \ref{lemma:jumptime} and \ref{lemma:Ybound2}, and
  (\ref{eq:errorestimate_high})-(\ref{eq:estimateBi3Bi4}), we can
  obtain the desired result.
\end{proof}

\begin{remark}
{\rm(I)} The convergent order is lower when $s$ is small. However, consider the following jump-adapted
 time discretization: Let $\Delta{t}=(T-t)/M$ be a mesh size, $M\in \mathbb{Z}^{+}$,
 \begin{equation}
  \tau_{i+1}=\inf\{t>\tau_{i}:\Delta{{L}_{t}^{\varepsilon}}\neq 0\}\wedge
 (\tau_{i}+\Delta{t})\wedge\overline{\tau}_{t,\texttt{{\rm{x}}}}\wedge{T}.
 \end{equation}
 Then we can get the following error estimate
 \begin{equation}
 \left|\mathbf{E}\left[
   u(T\wedge\overline{\tau}_{t,\texttt{{\rm{x}}}},
   \overline{\texttt{{\rm{X}}}}_{T\wedge\overline{\tau}_{t,\texttt{{\rm{x}}}}})
   \overline{Y}_{T\wedge\overline{\tau}_{t,\texttt{{\rm{x}}}}}
   +\overline{Z}_{T\wedge\overline{\tau}_{t,\texttt{{\rm{x}}}}}\right]
    -u(t,\texttt{{\rm{x}}})\right|
    \leq C_{1}(\varepsilon^{2s}\wedge\Delta{t})+C_{2}\varepsilon^{\lfloor\beta\rfloor-2s}.
 \end{equation}
 Thus, if we let $\Delta{t}=\varepsilon$, we can get the first order convergence
 for small $s$.
 \\
{\rm(II)} If coefficients satisfy Assumption {\rm II} instead of Assumption {\rm I},
the higher-order scheme will lose accuracy. At this time, we have the following estimate
 \begin{equation}
 \left|\mathbf{E}\left[
   u(T\wedge\overline{\tau}_{t,\texttt{{\rm{x}}}},
   \overline{\texttt{{\rm{X}}}}_{T\wedge\overline{\tau}_{t,\texttt{{\rm{x}}}}})
   \overline{Y}_{T\wedge\overline{\tau}_{t,\texttt{{\rm{x}}}}}
   +\overline{Z}_{T\wedge\overline{\tau}_{t,\texttt{{\rm{x}}}}}\right]
    -u(t,\texttt{{\rm{x}}})\right|
    \leq C_{1}\varepsilon^{s}+C_{2}\varepsilon^{\lfloor\beta\rfloor-2s}.
 \end{equation}
\end{remark}

\section{Numerical experiments}
  In the section, numerical examples are carried out by using jump-adapted scheme (\ref{eq:approSolution})
and higher-order scheme (\ref{eq:approSolution2}) on an i5-8250U CPU.
We approximate the expectation by Monte Carlo method, so that it will have statistical error
$\frac{1}{\sqrt{N}}$, where $N>0$ is the number of samples.

  Although the derived schemes for equation ({\ref{eq:terminalboundaryproblem}}) in
more than two space dimensions, they are still suitable for the two spacial-dimension case.

\begin{example}\label{example1}
 Let $\mathbb{D}$ be a unit ball in $\mathbb{R}^{n}$ centered at the origin. Consider the
 following fractional heat equation,
 \begin{equation}\label{eq:example1}
 \left\{
  \begin{aligned}
  &\frac{\partial{u}}{\partial{t}}-(-\Delta)^{s}u+f(t,\texttt{{\rm{x}}})=0,
  \quad &&(t,\texttt{{\rm{x}}})\in(0,T]\times \mathbb{D},
  \\[3pt]
  &u(T,\texttt{{\rm{x}}})=T(1-|\texttt{{\rm{x}}}|^2)^{1+s},
  \quad &&\texttt{{\rm{x}}}\in {\mathbb{D}},
  \\[3pt]
  &u(t,\texttt{{\rm{x}}})=0, \quad &&(t,\texttt{{\rm{x}}})\in[0,T]
  \times \mathbb{R}^{n}\setminus\mathbb{D},
  \end{aligned}
  \right.
 \end{equation}
 where $s\in(0,1)$, $n\geq 2$, $T=1$, and
 \begin{equation}
  f(t,\texttt{{\rm{x}}})=
  2^{2s}\Gamma(2+s)\Gamma \left({n}/{2}+s\right){\Gamma\left({n}/{2}\right)}^{-1}
  \left(1-\left(1+{2s}/{n}\right)|\texttt{{\rm{x}}}|^2\right)t
  -(1-|\texttt{{\rm{x}}}|^2)^{1+s}.
 \end{equation}
 The exact solution to {\rm(\ref{eq:example1})} is
 \begin{equation}
  u(t,\texttt{{\rm{x}}})=
     t(1-|\texttt{{\rm{x}}}|^2)^{1+s}.
 \end{equation}
\end{example}
We set $s=0.25$, $0.5$, $0.75$, the number of samples $N=10^4$,
$\varepsilon=\frac{1}{10}$, $\frac{1}{20}$,
$\frac{1}{40}$, $\frac{1}{80}$, $\frac{1}{160}$ for $2,\,3,\,4,\,10$ dimensional cases, and
$\varepsilon=\frac{1}{5}$, $\frac{1}{10}$, $\frac{1}{20}$,
$\frac{1}{40}$, $\frac{1}{80}$ for $100$ dimensional case. We evaluate $ u(t,\texttt{{\rm{x}}})$
with $t=0.5$, $\texttt{{\rm{x}}}=\frac{1}{n}\ast\textbf{ones}(n)$, $n=2,\,3, \,4,\, 10,\,100$.
Table \ref{exmp1tab1} gives errors $\Big|u(t,\texttt{{\rm{x}}})
-\frac{1}{N}\sum\limits_{j=1}^{N}u(T\wedge\widetilde{\tau}_{t,\texttt{{\rm{x}}}},
   \widetilde{\texttt{{\rm{X}}}}_{T\wedge\widetilde{\tau}_{t,\texttt{{\rm{x}}}}}^{j})
   \widetilde{Y}_{T\wedge\widetilde{\tau}_{t,\texttt{{\rm{x}}}}}^{j}
   +\widetilde{Z}_{T\wedge\widetilde{\tau}_{t,\texttt{{\rm{x}}}}}^{j}\Big|$ by
jump-adapted scheme (\ref{eq:approSolution}). The statistical error produced by Monte Carlo
method does not reflect the error of the method by the large number of samples $N=10^4$.
It is clear to observe that the computational time slowly increases with the growth of dimension.
In addition, many small jumps occur when index $s$ grows, which costs a lot of computational
time. Figure \ref{1fig1} shows the convergent order of the scheme (\ref{eq:approSolution}) for
$u(t,\texttt{{\rm{x}}})\in{C^{1,1+s}([0,T]\times \mathbb{R}^{n})}$ is $2s\wedge(1-s)$.
The numerical results are in good agreement with the theoretical analysis.

\begin{table}[!htbp]
 \centering
 \caption{The error in simulation, the average number of steps, and the computational
  time (secs.) for Example \ref{example1} by scheme (\ref{eq:approSolution}) are given.}
  \label{exmp1tab1}
 \begin{tabular}{|c|ccc|ccc|ccc|}
   \hline
                       \multicolumn{10}{|c|}{$n=2$}\\
   \hline
   $\varepsilon$        &$s=0.25$    &step        &time
                        &$s=0.5$     &step        &time
                        &$s=0.75$    &step        &time \\
   \hline
                $\frac{1}{10}$    &3.525E-02   &1.71       &0.39
                                  &1.452E-02   &6.52       &0.28
                                  &2.978E-02   &6.62       &0.53       \\[3pt]
                $\frac{1}{20}$    &2.914E-02   &2.20       &0.20
                                  &1.012E-02   &6.52       &0.62
                                  &2.262E-02   &16.6       &1.15       \\[3pt]
                $\frac{1}{40}$    &2.151E-02   &2.90       &0.25
                                  &7.725E-03   &12.5       &0.85
                                  &1.867E-02   &43.2       &2.90       \\[3pt]
                $\frac{1}{80}$    &1.560E-02   &3.94       &0.29
                                  &5.648E-03   &24.7       &1.67
                                  &1.431E-02   &117        &8.06       \\[3pt]
                $\frac{1}{160}$   &1.131E-02   &5.37       &0.40
                                  &4.101E-03   &48.5       &3.23
                                  &1.165E-02   &323        &21.3       \\[3pt]
  \hline
                       \multicolumn{10}{|c|}{$n=3$}\\
   \hline
   $\varepsilon$        &$s=0.25$    &step        &time
                        &$s=0.5$     &step        &time
                        &$s=0.75$    &step        &time \\
   \hline
                $\frac{1}{10}$    &5.679E-02   &1.86       &0.32
                                  &6.624E-03   &4.36       &0.67
                                  &9.215E-02   &8.86       &1.46       \\[3pt]
                $\frac{1}{20}$    &4.347E-02   &2.42       &0.45
                                  &3.252E-03   &8.20       &1.10
                                  &6.331E-02   &22.3       &2.82       \\[3pt]
                $\frac{1}{40}$    &3.171E-02   &3.24       &0.53
                                  &2.408E-03   &15.7       &2.07
                                  &4.439E-02   &59.5       &7.64       \\[3pt]
                $\frac{1}{80}$    &2.034E-02   &4.35       &0.62
                                  &1.403E-03   &31.2       &3.98
                                  &3.105E-02   &160        &20.2       \\[3pt]
                $\frac{1}{160}$   &1.484E-03   &5.96       &0.93
                                  &9.072E-04   &62.2       &7.45
                                  &2.639E-02   &435        &54.3       \\[3pt]
  \hline
                      \multicolumn{10}{|c|}{$n=4$}\\
   \hline
   $\varepsilon$        &$s=0.25$    &step        &time
                        &$s=0.5$     &step        &time
                        &$s=0.75$    &step        &time \\
   \hline
                $\frac{1}{10}$    &7.138E-02   &1.95       &1.06
                                  &1.338E-02   &4.80       &2.62
                                  &1.299E-01   &10.0       &4.64       \\[3pt]
                $\frac{1}{20}$    &5.090E-02   &2.56       &1.32
                                  &8.977E-03   &9.22       &4.48
                                  &8.291E-02   &25.6       &12.3       \\[3pt]
                $\frac{1}{40}$    &3.841E-02   &3.41       &1.64
                                  &6.184E-03   &17.9       &8.51
                                  &5.776E-02   &67.8       &29.0       \\[3pt]
                $\frac{1}{80}$    &2.710E-02   &4.63       &2.53
                                  &4.095E-03   &35.4       &16.3
                                  &3.494E-02   &182        &84.4       \\[3pt]
                $\frac{1}{160}$   &1.936E-02   &6.27       &3.04
                                  &2.901E-03   &70.0       &33.3
                                  &2.790E-02   &501        &229       \\[3pt]
  \hline
                      \multicolumn{10}{|c|}{$n=10$}\\
   \hline
   $\varepsilon$        &$s=0.25$    &step        &time
                        &$s=0.5$     &step        &time
                        &$s=0.75$    &step        &time \\
   \hline
                $\frac{1}{10}$    &8.282E-02   &2.14       &3.65
                                  &2.260E-02   &6.06       &10.0
                                  &2.199E-01   &12.2       &20.3       \\[3pt]
                $\frac{1}{20}$    &6.203E-02   &2.89       &5.03
                                  &1.850E-02   &11.6       &17.3
                                  &1.431E-01   &31.2       &47.5       \\[3pt]
                $\frac{1}{40}$    &4.790E-02   &3.91       &6.54
                                  &1.787E-02   &22.8       &37.1
                                  &9.988E-02   &82.0       &125       \\[3pt]
                $\frac{1}{80}$    &3.502E-02   &5.41       &8.40
                                  &1.224E-02   &45.0       &70.4
                                  &6.776E-02   &222        &340       \\[3pt]
                $\frac{1}{160}$   &2.550E-02   &7.48       &12.3
                                  &8.396E-03   &89.1       &138
                                  &5.639E-02   &602        &885       \\[3pt]
  \hline
                      \multicolumn{10}{|c|}{$n=100$}\\
   \hline
   $\varepsilon$        &$s=0.25$    &step        &time
                        &$s=0.5$     &step        &time
                        &$s=0.75$    &step        &time \\
   \hline
                $\frac{1}{5}$     &9.041E-02   &1.89       &31.0
                                  &8.495E-02   &3.59       &59.1
                                  &4.192E-01   &5.46       &91.0       \\[3pt]
                $\frac{1}{10}$    &7.590E-02   &2.55       &41.7
                                  &5.571E-02   &6.81       &111
                                  &2.597E-01   &13.0       &210       \\[3pt]
                $\frac{1}{20}$    &5.708E-02   &3.58       &60.8
                                  &4.111E-02   &13.3       &221
                                  &1.761E-01   &33.1       &552       \\[3pt]
                $\frac{1}{40}$    &4.014E-02   &4.88       &79.0
                                  &3.051E-02   &26.0       &419
                                  &1.299E-01   &87.9       &1423       \\[3pt]
                $\frac{1}{80}$    &2.805E-02   &6.88       &113
                                  &2.204E-02   &50.8       &831
                                  &1.032E-01   &235        &3935       \\[3pt]
  \hline
 \end{tabular}
\end{table}

\begin{figure}[htbp]
\centering

\subfigure[$s=0.25$]{
\begin{minipage}[t]{0.5\linewidth}
\centering
\includegraphics[width=1\linewidth]{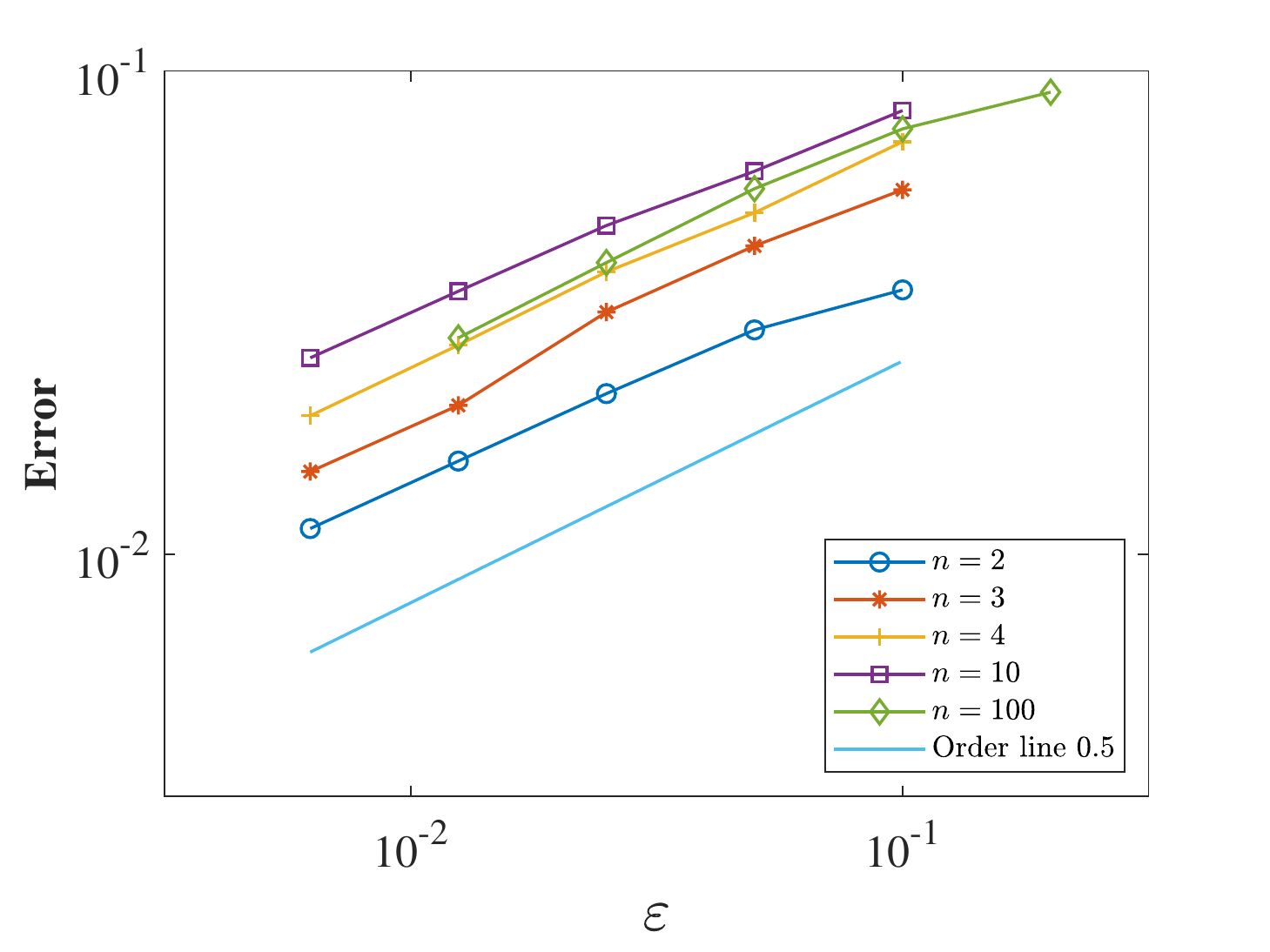}
\end{minipage}%
}%
\subfigure[$s=0.50$]{
\begin{minipage}[t]{0.5\linewidth}
\centering
\includegraphics[width=1\linewidth]{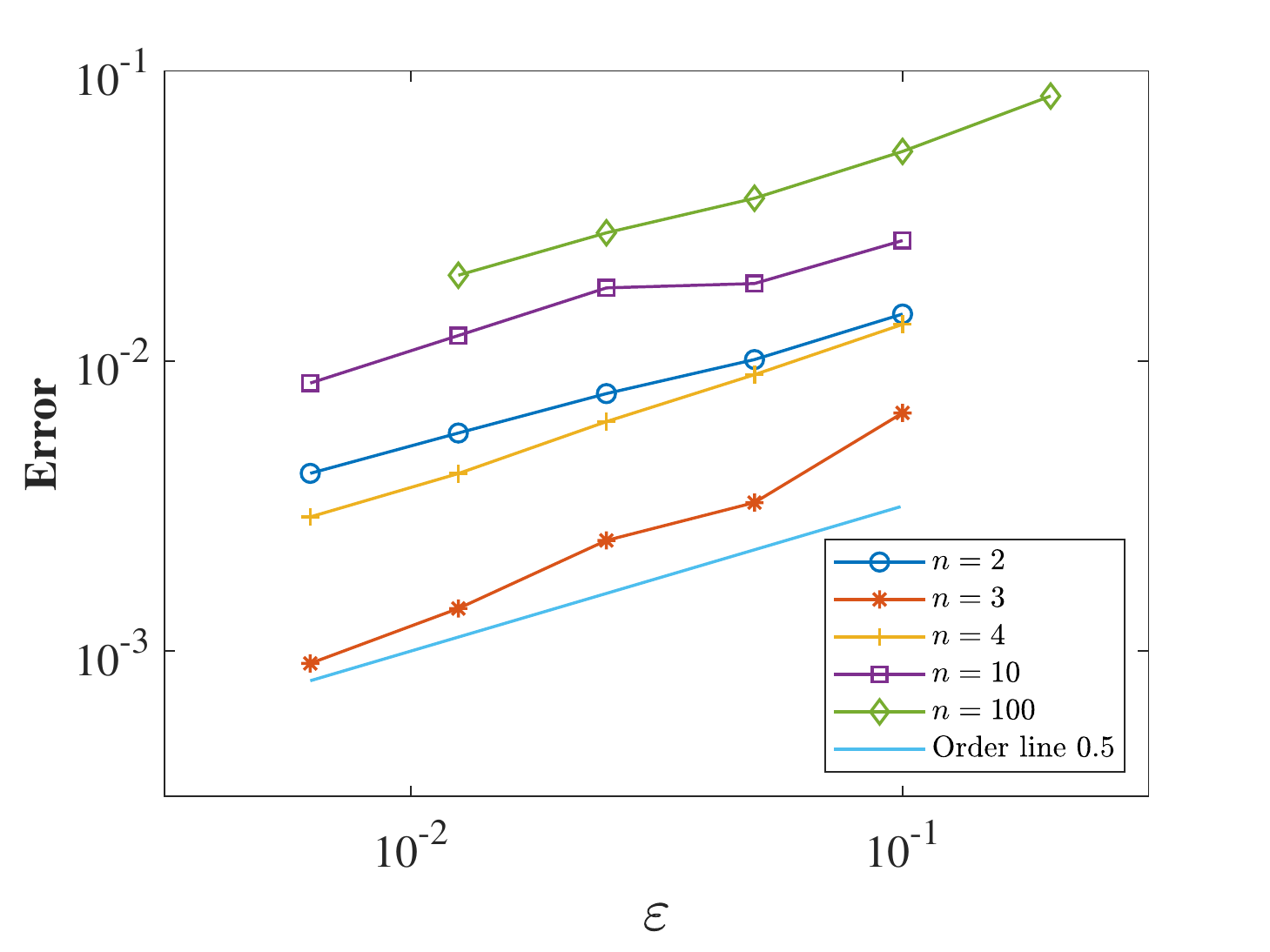}
\end{minipage}%
}%
  \\              
\subfigure[$s=0.75$]{
\begin{minipage}[t]{0.5\linewidth}
\centering
\includegraphics[width=1\linewidth]{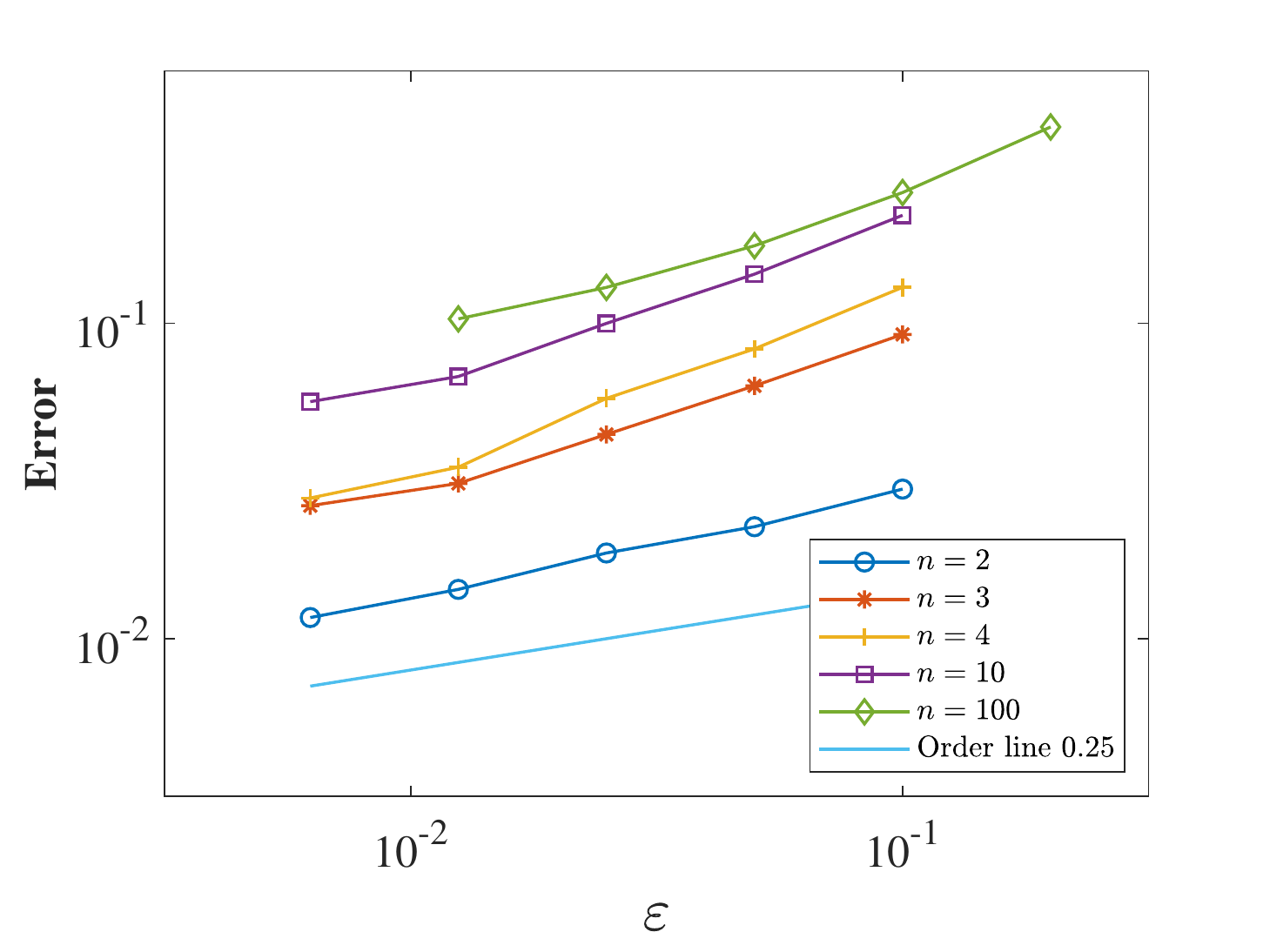}
\end{minipage}
}%
\centering
\captionsetup{font={footnotesize}}
\caption{Numerical errors of the scheme (\ref{eq:approSolution}) for Example \ref{example1}.
It shows that  the convergent order is $2s\wedge (1-s)$.} \label{1fig1}
\end{figure}

\begin{example}\label{example2}
 Let $\mathbb{D}$ be a unit ball in $\mathbb{R}^{n}$ centered at the origin. Consider the
 following parabolic problem with fractional Laplacian,
 \begin{equation}\label{eq:example2}
 \left\{
  \begin{aligned}
  &\frac{\partial{u}}{\partial{t}}-(-\Delta)^{s}u+\sum_{j=0}^{n}
  b_{j}(t,\texttt{{\rm{x}}})\frac{\partial{u}}{\partial{x_{j}}}+
  c(t,\texttt{{\rm{x}}})u+f(t,\texttt{{\rm{x}}})=0,
  \quad (t,\texttt{{\rm{x}}})\in[0,T)\times \mathbb{D},
  \\[3pt]
  &u(T,\texttt{{\rm{x}}})=T(1-|\texttt{{\rm{x}}}|^2)^{1+s}+T,
  \quad \texttt{{\rm{x}}}\in {\mathbb{D}},
  \\[3pt]
  &u(t,\texttt{{\rm{x}}})=t, \quad (t,\texttt{{\rm{x}}})\in[0,T]
  \times \mathbb{R}^{n}\setminus\mathbb{D},
  \end{aligned}
  \right.
 \end{equation}
 where $n\geq 2$, $T=1$,
 \begin{equation}
  \begin{aligned}
  b_{j}(t,\texttt{{\rm{x}}})&=t\sin(x_{j}),\hspace{67mm}
  \\
  c(t,\texttt{{\rm{x}}})&=\frac{e^{t}}{1+e^{-|x|}},
  \end{aligned}
  \end{equation}
  \noindent and
 \begin{equation}
  \begin{aligned}
 &f(t,\texttt{{\rm{x}}})=
  \frac{2^{2s}\Gamma(2+s)\Gamma \left({n}/{2}+s\right)}{\Gamma\left({n}/{2}\right)}
  \left(1-\left(1+{2s}/{n}\right)|\texttt{{\rm{x}}}|^2\right)t
  \\
  &+2t(1+s)(1-|\texttt{{\rm{x}}}|^2)^{s}\sum_{i}^{n}t\sin(x_{i})x_{i}
  \\
  &-\frac{e^{t}[t(1-|\texttt{{\rm{x}}}|^2)^{1+s}+t]}{1+e^{-|x|}}
  -(1-|\texttt{{\rm{x}}}|^2)^{1+s}-1.
  \end{aligned}
 \end{equation}
 The exact solution to {\rm(\ref{eq:example2})} is
 \begin{equation}
  u(t,\texttt{{\rm{x}}})=
     t(1-|\texttt{{\rm{x}}}|^2)^{1+s}+t.
 \end{equation}
\end{example}
We still set $s=0.25$, $0.5$, $0.75$, the number of samples $N=10^4$,
$\varepsilon=\frac{1}{10}$, $\frac{1}{20}$,
$\frac{1}{40}$, $\frac{1}{80}$, $\frac{1}{160}$ for $2,\,3,\,4,\,10$ dimensional cases, and
$\varepsilon=\frac{1}{5}$, $\frac{1}{10}$, $\frac{1}{20}$,
$\frac{1}{40}$, $\frac{1}{80}$ for $100$ dimensional case. We evaluate $ u(t,\texttt{{\rm{x}}})$
with $t=0.5$, $\texttt{{\rm{x}}}=\frac{1}{n}\ast\textbf{ones}(n)$, $n=2,\,3,\,4,\,10,\,100$.
The numerical results given in Table \ref{exmp2tab1} shows the efficiency of
the method (\ref{eq:approSolution}) and coincides with the theoretical analysis.
\begin{table}[!htbp]
 \centering
 \caption{The error in simulation, the average number of steps, and the computational time (secs.)
  for Example \ref{example2} by scheme (\ref{eq:approSolution}) are given.}\label{exmp2tab1}
 \begin{tabular}{|c|ccc|ccc|ccc|}
   \hline
                       \multicolumn{10}{|c|}{$n=2$}\\
   \hline
   $\varepsilon$        &$s=0.25$    &step        &time
                        &$s=0.5$     &step        &time
                        &$s=0.75$    &step        &time \\
   \hline
                $\frac{1}{10}$    &4.844E-02   &1.40       &0.39
                                  &2.966E-02   &3.06       &0.68
                                  &3.635E-02   &5.69       &1.12       \\[3pt]
                $\frac{1}{20}$    &3.345E-02   &1.86       &0.46
                                  &1.699E-02   &5.65       &1.15
                                  &2.388E-02   &14.5       &2.75       \\[3pt]
                $\frac{1}{40}$    &2.367E-02   &2.48       &0.59
                                  &4.984E-03   &10.8       &2.06
                                  &1.639E-02   &38.1       &7.29       \\[3pt]
                $\frac{1}{80}$    &1.693E-02   &3.41       &0.75
                                  &1.348E-03   &21.2       &3.93
                                  &1.349E-02   &103        &19.1       \\[3pt]
                $\frac{1}{160}$   &1.192E-02   &4.69       &0.96
                                  &9.345E-03   &42.3       &7.90
                                  &1.110E-02   &281        &54.3       \\[3pt]
  \hline
                       \multicolumn{10}{|c|}{$n=3$}\\
   \hline
   $\varepsilon$        &$s=0.25$    &step        &time
                        &$s=0.5$     &step        &time
                        &$s=0.75$    &step        &time \\
   \hline
                $\frac{1}{10}$    &1.045E-02   &1.67       &0.75
                                  &3.756E-03   &3.88       &1.50
                                  &1.230E-01   &7.74       &2.59       \\[3pt]
                $\frac{1}{20}$    &1.382E-02   &2.18       &0.90
                                  &1.837E-02   &7.31       &2.48
                                  &7.749E-02   &20.0       &6.37       \\[3pt]
                $\frac{1}{40}$    &1.880E-02   &2.98       &1.21
                                  &5.084E-03   &14.2       &4.70
                                  &3.775E-02   &53.3       &16.8       \\[3pt]
                $\frac{1}{80}$    &1.379E-02   &4.13       &1.64
                                  &5.084E-03   &28.3       &9.71
                                  &2.674E-02   &145        &45.2       \\[3pt]
                $\frac{1}{160}$   &1.025E-02   &5.63       &2.15
                                  &3.438E-03   &55.7       &18.8
                                  &2.253E-02   &393        &125       \\[3pt]
  \hline
                      \multicolumn{10}{|c|}{$n=4$}\\
   \hline
   $\varepsilon$        &$s=0.25$    &step        &time
                        &$s=0.5$     &step        &time
                        &$s=0.75$    &step        &time \\
   \hline
                $\frac{1}{10}$    &9.976E-03   &1.80       &1.71
                                  &3.682E-02   &4.37       &3.89
                                  &1.551E-01   &8.98       &7.98       \\[3pt]
                $\frac{1}{20}$    &1.905E-02   &2.38       &2.20
                                  &1.703E-02   &8.44       &7.03
                                  &9.613E-02   &23.2       &19.7       \\[3pt]
                $\frac{1}{40}$    &1.240E-02   &3.23       &2.98
                                  &9.508E-03   &16.4       &14.0
                                  &6.017E-02   &61.5       &61.7       \\[3pt]
                $\frac{1}{80}$    &8.541E-03   &4.42       &4.06
                                  &6.951E-03   &32.2       &26.8
                                  &3.844E-02   &167        &145       \\[3pt]
                $\frac{1}{160}$   &6.875E-03   &6.11       &5.46
                                  &4.833E-03   &64.4       &53.9
                                  &2.759E-02   &455        &394       \\[3pt]
  \hline
                      \multicolumn{10}{|c|}{$n=10$}\\
   \hline
   $\varepsilon$        &$s=0.25$    &step        &time
                        &$s=0.5$     &step        &time
                        &$s=0.75$    &step        &time \\
   \hline
                $\frac{1}{10}$    &1.364E-02   &2.05       &3.56
                                  &4.021E-02   &5.50       &9.14
                                  &2.227E-01   &11.3       &18.7       \\[3pt]
                $\frac{1}{20}$    &1.050E-02   &2.76       &4.59
                                  &2.478E-02   &10.7       &18.7
                                  &1.441E-01   &28.9       &46.7       \\[3pt]
                $\frac{1}{40}$    &7.874E-03   &3.80       &7.04
                                  &1.374E-02   &21.0       &35.4
                                  &9.340E-02   &77.2       &131       \\[3pt]
                $\frac{1}{80}$    &4.314E-03   &5.19       &8.57
                                  &8.585E-03   &41.9       &69.9
                                  &7.153E-02   &207        &344       \\[3pt]
                $\frac{1}{160}$   &2.094E-03   &7.10       &11.9
                                  &6.186E-03   &82.9       &139
                                  &5.607E-02   &563        &931       \\[3pt]
  \hline
                     \multicolumn{10}{|c|}{$n=100$}\\
   \hline
   $\varepsilon$        &$s=0.25$    &step        &time
                        &$s=0.5$     &step        &time
                        &$s=0.75$    &step        &time \\
   \hline
                $\frac{1}{5}$     &2.543E-02   &1.79       &31.3
                                  &8.199E-02   &3.44       &60.5
                                  &4.181E-01   &5.38       &91.5       \\[3pt]
                $\frac{1}{10}$    &1.774E-02   &2.49       &45.2
                                  &5.281E-02   &6.56       &111
                                  &2.554E-01   &12.8       &216       \\[3pt]
                $\frac{1}{20}$    &1.264E-02   &3.43       &58.5
                                  &3.641E-02   &12.8       &219
                                  &1.784E-01   &33.0       &1491       \\[3pt]
                $\frac{1}{40}$    &7.963E-03   &4.82       &82.3
                                  &2.768E-02   &24.8       &433
                                  &1.268E-01   &86.7       &1490       \\[3pt]
                $\frac{1}{80}$    &5.403E-03   &6.69       &115
                                  &1.977E-02   &49.2       &827
                                  &9.952E-02   &234        &4039       \\[3pt]
  \hline
 \end{tabular}
\end{table}

\begin{example}\label{example3}
 Let $\mathbb{D}$ be a unit ball in $\mathbb{R}^{n}$ centered at the origin. Consider the
 following fractional heat equation,
 \begin{equation}\label{eq:example3}
 \left\{
  \begin{aligned}
  &\frac{\partial{u_{i}}}{\partial{t}}-(-\Delta)^{s}u_{i}
  +f_{i}(t,\texttt{{\rm{x}}})=0,
  \quad &&(t,\texttt{{\rm{x}}})\in[0,T)\times \mathbb{D},
  \\[3pt]
  &u_{i}(T,\texttt{{\rm{x}}})=T(1-|\texttt{{\rm{x}}}|^2)^{1+i+s},
  \quad &&\texttt{{\rm{x}}}\in {\mathbb{D}},
  \\[3pt]
  &u_{i}(t,\texttt{{\rm{x}}})=0, \quad &&(t,\texttt{{\rm{x}}})\in[0,T]
  \times \mathbb{R}^{n}\setminus\mathbb{D},
  \end{aligned}
  \right.
 \end{equation}
 where $i=1,\,2$, $n\geq 2$, and
 \begin{equation}\label{eq:rhs4example3}
  \begin{aligned}
  f_{i}(t,\texttt{{\rm{x}}})=-(1-|\texttt{{\rm{x}}}|^2)^{2+s}
  -\frac{C(n,s)B(-s,i+s+2)\pi^{\frac{n}{2}}}{\Gamma\left(\frac{n}{2}\right)}
  {_2}F_{1}\left(s+\frac{n}{2},-i-1;\frac{n}{2};|\texttt{{\rm{x}}}|^2\right)t.
   \end{aligned}
  \end{equation}
 Here $C(n,s)$ is given in equation {\rm (\ref{eq:coeffient})}
 and ${_2F_{1}(a,b;c;z)}$ is the hypergeometric function.
 The exact solution to {\rm(\ref{eq:example3})} is
 \begin{equation}
  u_{1}(t,\texttt{{\rm{x}}})=t(1-|\texttt{{\rm{x}}}|^2)^{2+s}, \quad i=1,
 \end{equation}
 and
 \begin{equation}
  u_{2}(t,\texttt{{\rm{x}}})=t(1-|\texttt{{\rm{x}}}|^2)^{3+s}, \quad i=2.
 \end{equation}
\end{example}
 We set $s=0.25$, $0.5$, $0.75$, $\varepsilon=\frac{1}{4}$, $\frac{1}{8}$,
$\frac{1}{16}$, $\frac{1}{32}$, $\frac{1}{64}$ for $2,\,3,\,4,\,10$ dimensional cases and
$\varepsilon=\frac{1}{2}$, $\frac{1}{4}$, $\frac{1}{8}$,
$\frac{1}{16}$, $\frac{1}{32}$ for $100$ dimensional case.
We evaluate $ u(t,\texttt{{\rm{x}}})$
with $t=0.5$, $\texttt{{\rm{x}}}=\frac{1}{n}\ast\textbf{ones}(n)$, $n=2,\,3, \,4,\, 10,\,100$.
To avoid statistical error, we set $N=4\times10^4$
for the case of $s=0.75,\, \varepsilon=\frac{1}{32}$ and $N=3\times10^5$
for the case of $s=0.75,\, \varepsilon=\frac{1}{64}$. In other cases, $N$ is still set by $10^4$.
Tables \ref{exmp3tab1} and \ref{exmp3tab2} give numerical errors
$\Big|u(t,\texttt{{\rm{x}}})-\frac{1}{N}
\sum\limits_{j=1}^{N}u(T\wedge\overline{\tau}_{t,\texttt{{\rm{x}}}},
   \overline{\texttt{{\rm{X}}}}_{T\wedge\overline{\tau}_{t,\texttt{{\rm{x}}}}}^{j})
   \overline{Y}_{T\wedge\overline{\tau}_{t,\texttt{{\rm{x}}}}}^{j}
   +\overline{Z}_{T\wedge\overline{\tau}_{t,\texttt{{\rm{x}}}}}^{j}\Big|$
with $f_{1}(t,\texttt{{\rm{x}}})$ and $f_{2}(t,\texttt{{\rm{x}}})$ in (\ref{eq:rhs4example3})
by higher-order jump-adapted scheme (\ref{eq:approSolution2}), respectively.
Compared with steps and computational time in Table \ref{exmp3tab1},
the corresponding results in Table \ref{exmp3tab2} almost shows no evident differences.
This is due to the fact that being given different $f_{i}(t,\texttt{{\rm{x}}})$ in equation
(\ref{eq:example3}) with the same index $s$  will not change the
trajectory of  L\'{e}vy processes. Figures (a), (b), (c) in Figure \ref{2fig1}
and figures (d), (e), (f) in the same figure show the convergent order of the high-order
scheme (\ref{eq:approSolution2}) with different functions
$f_{i}(t,\texttt{{\rm{x}}}),\,i=1,\,2$ in equation (\ref{eq:example3}),
respectively. When $s=0.25,0.50$, figures (a), (b), (d), and (e) demonstrate
the convergent order for equation (\ref{eq:example3}) with $f_{1}(t,\texttt{{\rm{x}}})$
and $f_{2}(t,\texttt{{\rm{x}}})$ are the same, which coincides with the theoretical analysis.
Compared with figure (c) for $s=0.75$, figure (f) shows higher convergent order due to the higher
regularity of the $u(t,\texttt{{\rm{x}}})$.

\begin{table}[!htbp]
 \centering
 \caption{The error in simulation, the average number of steps, and the computational time (secs.)
  for $i=1$ of Example \ref{example3} by scheme (\ref{eq:approSolution2}) are given.}\label{exmp3tab1}
 \begin{tabular}{|c|ccc|ccc|ccc|}
   \hline
                       \multicolumn{10}{|c|}{$n=2$}\\
   \hline
   $\varepsilon$        &$s=0.25$    &step        &time
                        &$s=0.5$     &step        &time
                        &$s=0.75$    &step        &time \\
   \hline
                $\frac{1}{4}$     &1.620E-02   &1.26       &0.60
                                  &1.248E-02   &1.55       &0.90
                                  &2.167E-02   &1.54       &0.93       \\[3pt]
                $\frac{1}{8}$     &1.484E-02   &1.57       &0.76
                                  &1.122E-02   &2.78       &1.26
                                  &5.954E-03   &3.79       &1.79       \\[3pt]
                $\frac{1}{16}$    &1.060E-02   &2.02       &0.92
                                  &7.045E-03   &5.18       &2.42
                                  &4.236E-03   &10.2       &4.78       \\[3pt]
                $\frac{1}{32}$    &7.483E-03   &2.66       &1.20
                                  &2.885E-03   &10.0       &4.41
                                  &2.786E-03   &27.8       &11.8       \\[3pt]
                $\frac{1}{64}$    &5.493E-03   &3.54       &1.60
                                  &1.528E-03   &19.7       &8.23
                                  &2.090E-03   &76.4       &32.2       \\[3pt]
  \hline
                       \multicolumn{10}{|c|}{$n=3$}\\
   \hline
   $\varepsilon$        &$s=0.25$    &step        &time
                        &$s=0.5$     &step        &time
                        &$s=0.75$    &step        &time \\
   \hline
                $\frac{1}{4}$     &4.228E-02   &1.34       &0.79
                                  &1.447E-02   &1.86       &1.04
                                  &2.280E-02   &1.93       &1.29       \\[3pt]
                $\frac{1}{8}$     &3.200E-02   &1.72       &1.14
                                  &7.239E-03   &3.36       &1.96
                                  &1.114E-02   &5.04       &2.70       \\[3pt]
                $\frac{1}{16}$    &2.166E-02   &2.22       &1.34
                                  &4.424E-03   &6.48       &3.25
                                  &5.062E-02   &13.5       &7.06       \\[3pt]
                $\frac{1}{32}$    &1.535E-02   &2.93       &1.54
                                  &2.255E-03   &12.6       &6.26
                                  &3.725E-02   &38.0       &18.9       \\[3pt]
                $\frac{1}{64}$    &1.085E-02   &3.97       &2.12
                                  &1.166E-03   &24.6       &12.1
                                  &2.475E-02   &105        &52.0       \\[3pt]
  \hline
                      \multicolumn{10}{|c|}{$n=4$}\\
   \hline
   $\varepsilon$        &$s=0.25$    &step        &time
                        &$s=0.5$     &step        &time
                        &$s=0.75$    &step        &time \\
   \hline
                $\frac{1}{4}$     &5.372E-02   &1.39       &1.20
                                  &1.128E-02   &2.02       &2.04
                                  &3.068E-02   &2.12       &1.93       \\[3pt]
                $\frac{1}{8}$     &3.994E-02   &1.78       &1.50
                                  &6.274E-03   &3.76       &3.14
                                  &1.481E-02   &5.62       &4.65       \\[3pt]
                $\frac{1}{16}$    &2.420E-02   &2.31       &1.90
                                  &1.870E-03   &7.24       &5.92
                                  &2.992E-03   &15.1       &12.3       \\[3pt]
                $\frac{1}{32}$    &1.615E-02   &3.11       &2.54
                                  &1.165E-03   &14.2       &11.3
                                  &2.309E-02   &43.2       &37.9       \\[3pt]
                $\frac{1}{64}$    &1.070E-03   &4.18       &3.42
                                  &5.535E-04   &27.8       &22.6
                                  &1.729E-03   &121        &98.1       \\[3pt]
  \hline
                      \multicolumn{10}{|c|}{$n=10$}\\
   \hline
   $\varepsilon$        &$s=0.25$    &step        &time
                        &$s=0.5$     &step        &time
                        &$s=0.75$    &step        &time \\
   \hline
                $\frac{1}{4}$     &7.687E-02   &1.49       &3.45
                                  &9.497E-02   &2.39       &5.14
                                  &1.079E-01   &2.36       &4.82       \\[3pt]
                $\frac{1}{8}$     &4.493E-02   &1.97       &3.78
                                  &3.861E-03   &4.67       &9.03
                                  &3.802E-02   &6.51       &12.9       \\[3pt]
                $\frac{1}{16}$    &1.723E-02   &2.61       &4.96
                                  &9.114E-03   &9.20       &17.5
                                  &2.188E-02   &18.3       &34.3       \\[3pt]
                $\frac{1}{32}$    &1.305E-03   &3.54       &6.81
                                  &4.943E-03   &18.1       &37.8
                                  &1.577E-02   &51.4       &95.9       \\[3pt]
                $\frac{1}{64}$    &9.910E-03   &4.90       &9.76
                                  &2.287E-03   &35.6       &66.3
                                  &1.154E-02   &143        &267       \\[3pt]
  \hline
                      \multicolumn{10}{|c|}{$n=100$}\\
   \hline
   $\varepsilon$        &$s=0.25$    &step        &time
                        &$s=0.5$     &step        &time
                        &$s=0.75$    &step        &time \\
   \hline
                $\frac{1}{2}$     &1.110E-01   &1.26       &21.6
                                  &1.996E-01   &1.28       &21.6
                                  &5.333E-01   &0.85       &16.5       \\[3pt]
                $\frac{1}{4}$     &1.048E-01   &1.70       &28.5
                                  &4.635E-02   &2.56       &44.6
                                  &2.233E-01   &2.45       &41.1       \\[3pt]
                $\frac{1}{8}$     &7.995E-02   &2.34       &40.8
                                  &2.198E-02   &5.18       &87.1
                                  &6.857E-02   &6.92       &115       \\[3pt]
                $\frac{1}{16}$    &5.291E-02   &3.21       &56.2
                                  &1.815E-02   &10.2       &169
                                  &3.631E-02   &19.6       &338       \\[3pt]
                $\frac{1}{32}$    &3.729E-02   &4.49       &77.7
                                  &1.242E-02   &20.1       &350
                                  &2.395E-02   &54.0       &922       \\[3pt]
  \hline
 \end{tabular}
\end{table}

\begin{table}[!htbp]
 \centering
 \caption{The error in simulation, the average number of steps, and the computational time (secs.)
  for $i=2$ of Example \ref{example3} by scheme (\ref{eq:approSolution2}) are given.}\label{exmp3tab2}
 \begin{tabular}{|c|ccc|ccc|ccc|}
   \hline
                       \multicolumn{10}{|c|}{$n=2$}\\
   \hline
   $\varepsilon$        &$s=0.25$    &step        &time
                        &$s=0.5$     &step        &time
                        &$s=0.75$    &step        &time \\
   \hline
                $\frac{1}{4}$     &3.680E-03   &1.26       &0.67
                                  &7.855E-03   &1.54       &0.78
                                  &3.315E-02   &1.55       &0.90       \\[3pt]
                $\frac{1}{8}$     &3.143E-03   &1.58       &0.82
                                  &6.639E-03   &2.79       &1.78
                                  &1.182E-02   &3.82       &1.81       \\[3pt]
                $\frac{1}{16}$    &2.219E-03   &2.03       &0.92
                                  &3.624E-03   &5.22       &2.60
                                  &4.429E-03   &10.2       &4.65       \\[3pt]
                $\frac{1}{32}$    &1.639E-03   &2.66       &1.28
                                  &2.364E-03   &10.0       &4.68
                                  &1.483E-03   &27.8       &49.7       \\[3pt]
                $\frac{1}{64}$    &1.203E-03   &3.54       &1.56
                                  &1.159E-03   &19.6       &8.17
                                  &5.387E-04   &77.2       &984       \\[3pt]
  \hline
                       \multicolumn{10}{|c|}{$n=3$}\\
   \hline
   $\varepsilon$        &$s=0.25$    &step        &time
                        &$s=0.5$     &step        &time
                        &$s=0.75$    &step        &time \\
   \hline
                $\frac{1}{4}$     &2.620E-02   &1.35       &0.78
                                  &6.874E-03   &1.86       &1.09
                                  &3.593E-03   &1.89       &1.21       \\[3pt]
                $\frac{1}{8}$     &2.263E-02   &1.69       &1.03
                                  &1.400E-03   &3.40       &1.84
                                  &6.062E-03   &5.04       &2.70       \\[3pt]
                $\frac{1}{16}$    &1.532E-02   &2.20       &1.18
                                  &6.189E-04   &6.50       &3.51
                                  &2.121E-03   &13.7       &7.37       \\[3pt]
                $\frac{1}{32}$    &1.057E-02   &2.92       &1.56
                                  &2.450E-04   &12.5       &6.32
                                  &6.192E-04   &37.7       &74.0       \\[3pt]
                $\frac{1}{64}$    &7.335E-03   &3.92       &2.04
                                  &9.949E-05   &24.8       &12.4
                                  &2.255E-04   &105        &1549       \\[3pt]
  \hline
                      \multicolumn{10}{|c|}{$n=4$}\\
   \hline
   $\varepsilon$        &$s=0.25$    &step        &time
                        &$s=0.5$     &step        &time
                        &$s=0.75$    &step        &time \\
   \hline
                $\frac{1}{4}$     &4.448E-02   &1.39       &1.23
                                  &5.686E-03   &1.99       &1.75
                                  &3.241E-02   &2.12       &1.96       \\[3pt]
                $\frac{1}{8}$     &3.190E-02   &1.78       &1.60
                                  &2.909E-03   &3.77       &3.56
                                  &1.197E-02   &5.69       &4.68       \\[3pt]
                $\frac{1}{16}$    &1.800E-02   &2.30       &1.92
                                  &1.392E-03   &7.22       &6.28
                                  &3.720E-03   &15.3       &12.7       \\[3pt]
                $\frac{1}{32}$    &1.244E-02   &3.08       &2.56
                                  &7.193E-04   &14.1       &11.2
                                  &8.106E-04   &43.0       &136       \\[3pt]
                $\frac{1}{64}$    &8.068E-03   &4.19       &3.43
                                  &3.245E-04   &28.0       &22.6
                                  &2.650E-04   &120        &2913       \\[3pt]
  \hline
                      \multicolumn{10}{|c|}{$n=10$}\\
   \hline
   $\varepsilon$        &$s=0.25$    &step        &time
                        &$s=0.5$     &step        &time
                        &$s=0.75$    &step        &time \\
   \hline
                $\frac{1}{2}$     &1.087E-01   &1.18       &2.46
                                  &1.035E-01   &1.25       &2.57
                                  &1.239E-01   &0.81       &2.25       \\[3pt]
                $\frac{1}{4}$     &7.280E-02   &1.50       &3.09
                                  &1.622E-02   &2.39       &4.76
                                  &1.569E-01   &2.33       &4.76       \\[3pt]
                $\frac{1}{8}$     &4.345E-02   &1.96       &4.06
                                  &8.147E-03   &4.62       &8.89
                                  &5.641E-03   &6.54       &13.0       \\[3pt]
                $\frac{1}{16}$    &2.120E-02   &2.63       &5.25
                                  &5.924E-03   &9.03       &18.0
                                  &2.259E-02   &18.3       &38.5       \\[3pt]
                $\frac{1}{32}$    &6.903E-03   &3.55       &6.82
                                  &6.192E-03   &18.1       &34.4
                                  &1.600E-02   &50.7       &384       \\[3pt]
                $\frac{1}{64}$    &4.523E-03   &4.80       &9.32
                                  &2.287E-03   &35.7       &67.6
                                  &7.617E-02   &142        &8247       \\[3pt]
  \hline
                      \multicolumn{10}{|c|}{$n=100$}\\
   \hline
   $\varepsilon$        &$s=0.25$    &step        &time
                        &$s=0.5$     &step        &time
                        &$s=0.75$    &step        &time \\
   \hline
                $\frac{1}{2}$     &9.560E-01   &1.25       &20.7
                                  &2.831E-01   &1.29       &21.0
                                  &7.816E-01   &0.84       &15.5       \\[3pt]
                $\frac{1}{4}$     &9.787E-02   &1.68       &28.8
                                  &7.095E-02   &2.57       &43.5
                                  &7.816E-01   &2.42       &43.0       \\[3pt]
                $\frac{1}{8}$     &8.131E-02   &2.32       &40.7
                                  &3.068E-02   &5.17       &87.0
                                  &8.823E-02   &6.91       &119       \\[3pt]
                $\frac{1}{16}$    &5.694E-02   &3.21       &53.2
                                  &2.482E-02   &10.2       &181
                                  &4.192E-02   &19.6       &327       \\[3pt]
                $\frac{1}{32}$    &3.867E-02   &4.41       &73.6
                                  &1.373E-02   &20.0       &338
                                  &2.650E-02   &54.6       &3679       \\[3pt]
  \hline
 \end{tabular}
\end{table}

\begin{figure}[htbp]
\centering
\subfigure[$i=1,\,s=0.25$]{
\begin{minipage}[t]{0.5\linewidth}
\centering
\includegraphics[width=1\linewidth]{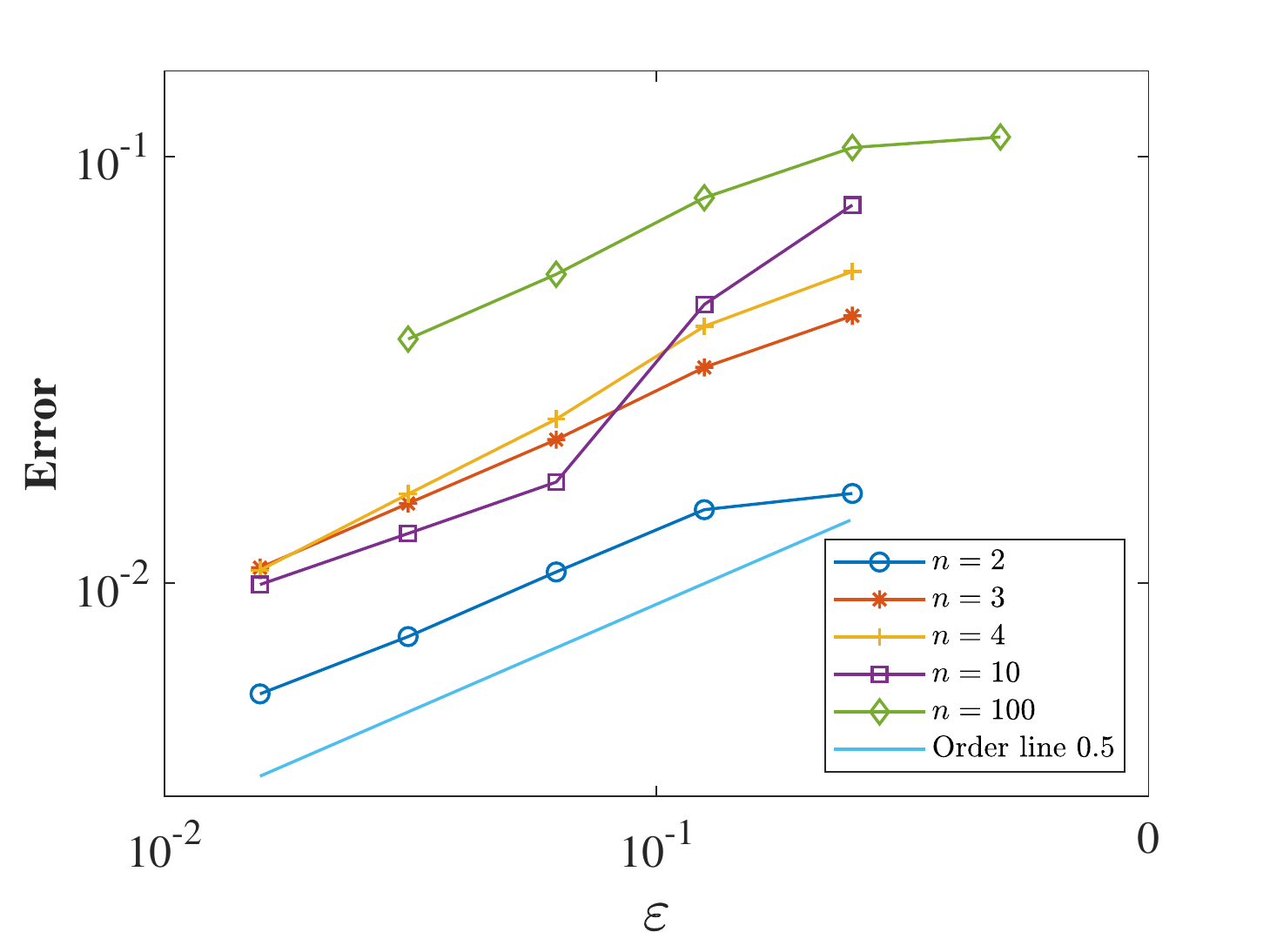}
\end{minipage}%
}%
\subfigure[$i=1,\,s=0.50$]{
\begin{minipage}[t]{0.5\linewidth}
\centering
\includegraphics[width=1\linewidth]{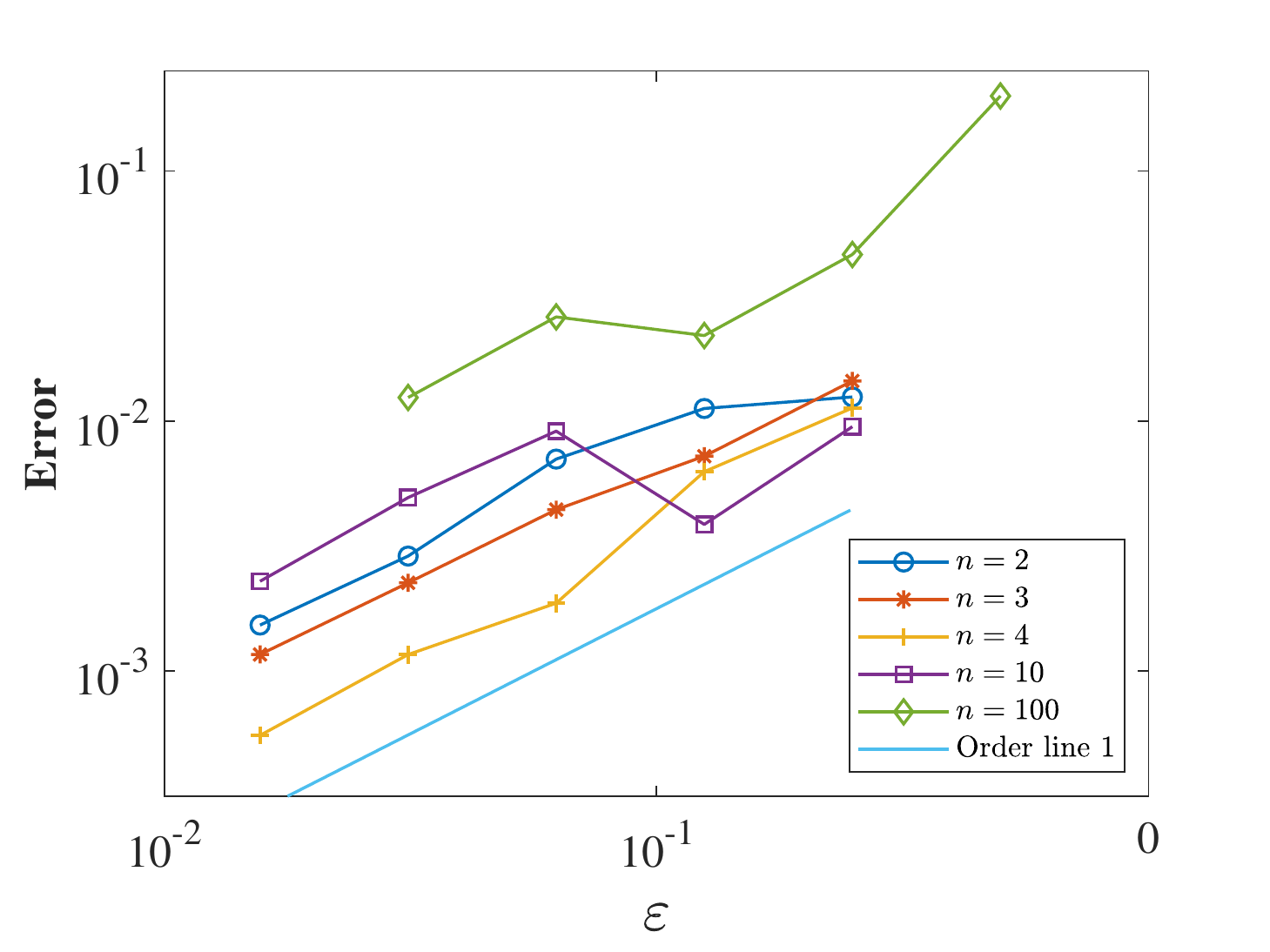}
\end{minipage}%
}%
  \\              
\subfigure[$i=1,\,s=0.75$]{
\begin{minipage}[t]{0.5\linewidth}
\centering
\includegraphics[width=1\linewidth]{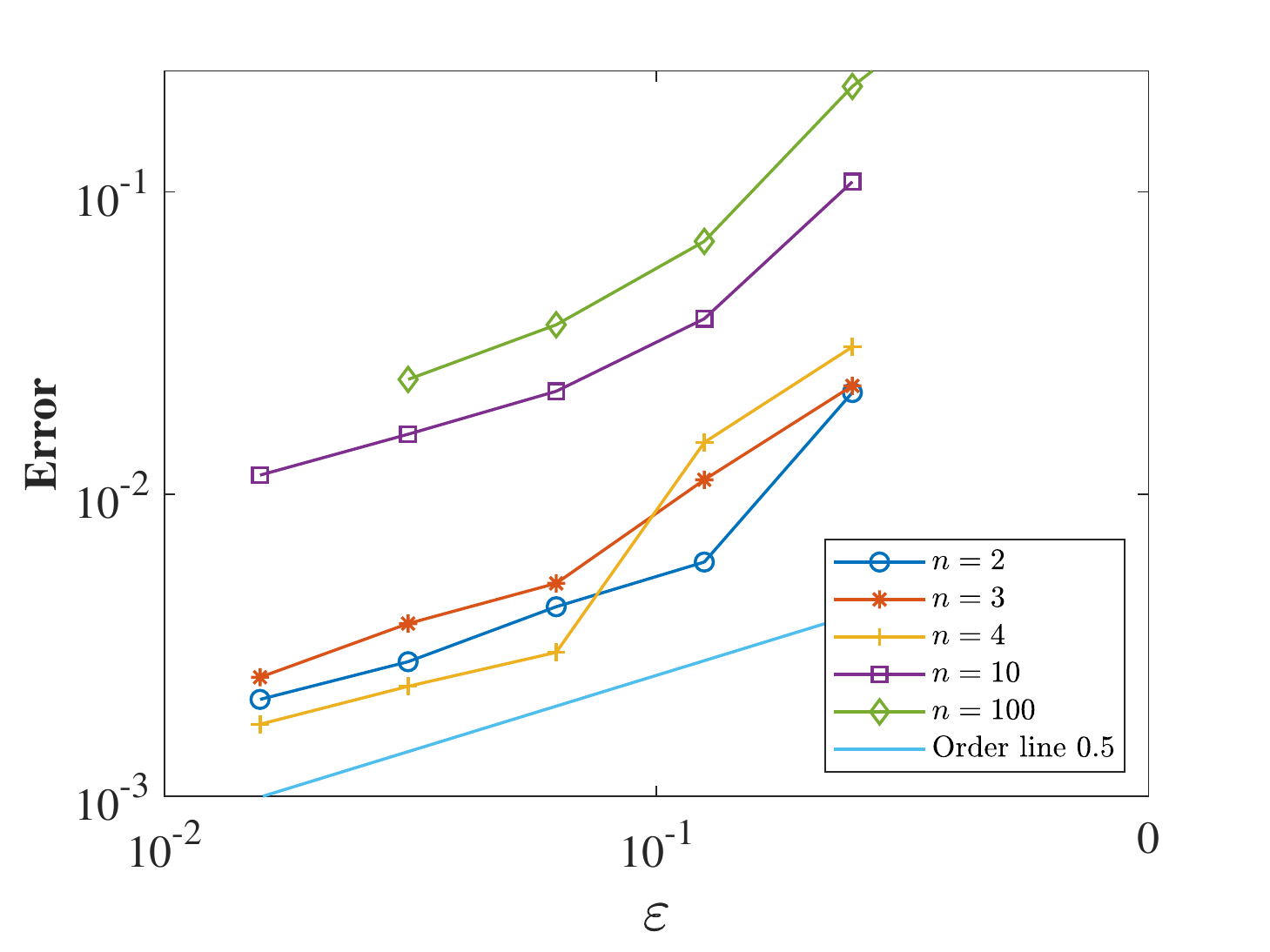}
\end{minipage}
}%
\subfigure[$i=2,\,s=0.25$]{
\begin{minipage}[t]{0.5\linewidth}
\centering
\includegraphics[width=1\linewidth]{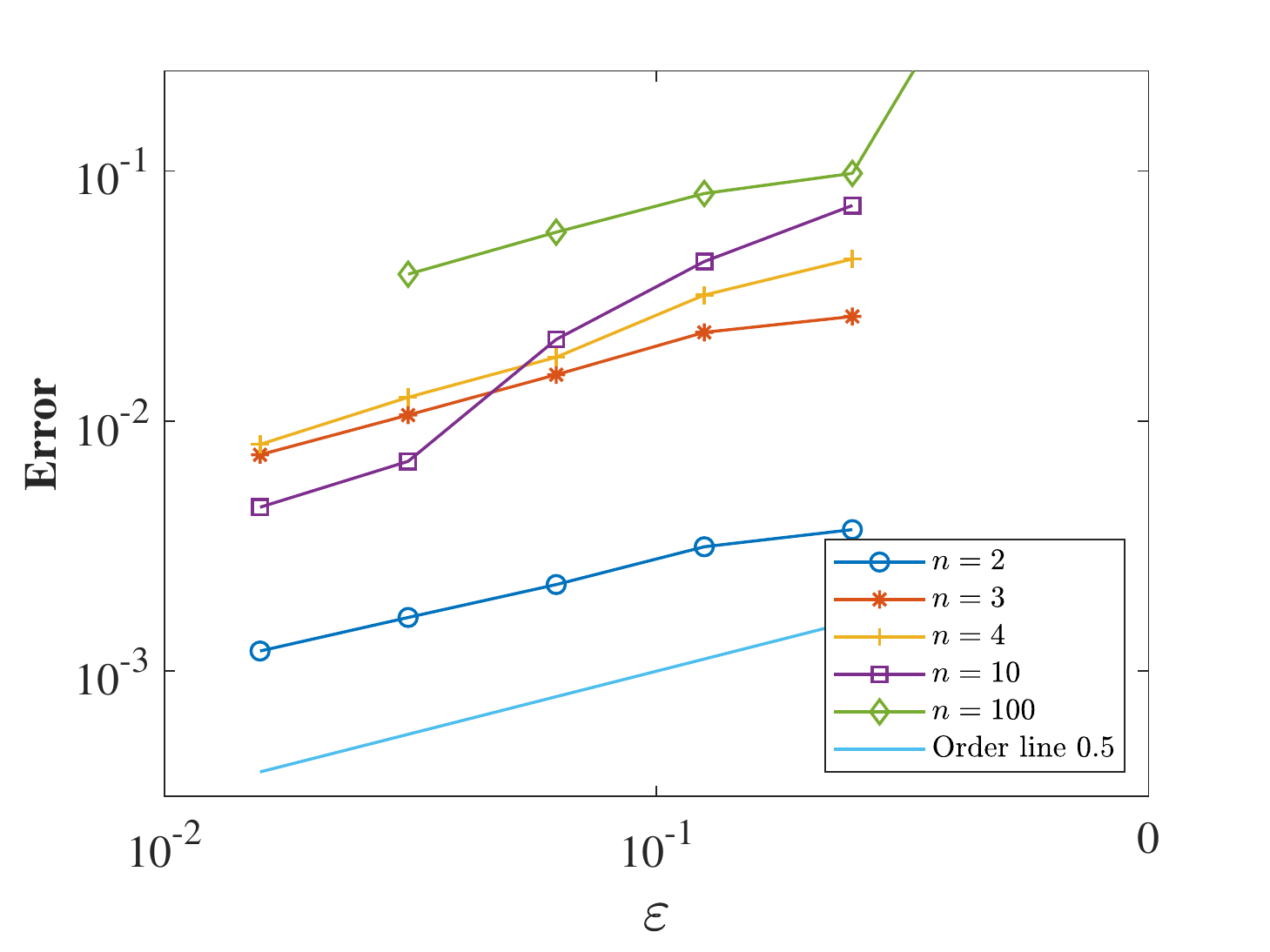}
\end{minipage}
}
\\
\subfigure[$i=2,\,s=0.5$]{
\begin{minipage}[t]{0.5\linewidth}
\centering
\includegraphics[width=1\linewidth]{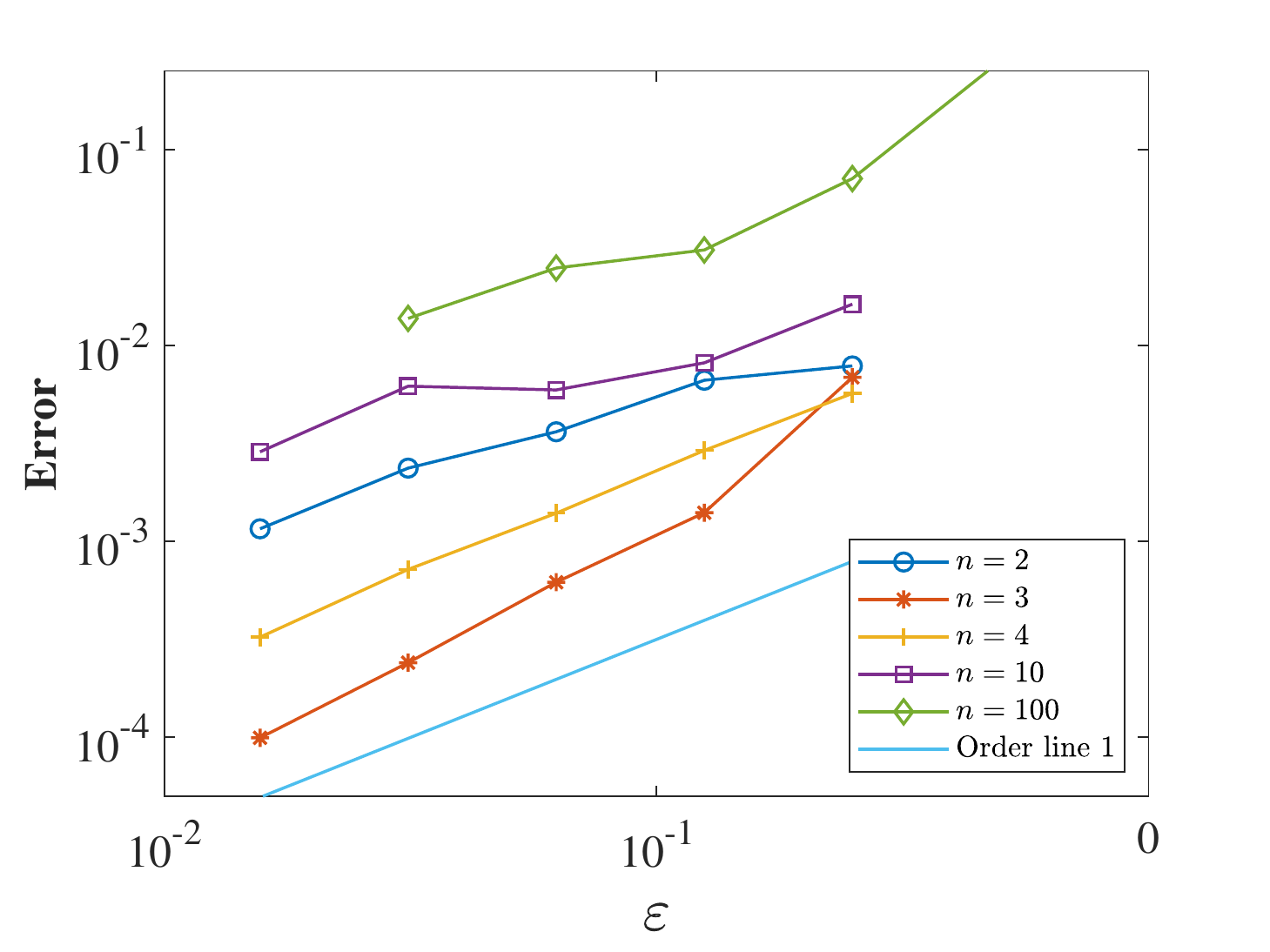}
\end{minipage}
}%
\subfigure[$i=2,\,s=0.75$]{
\begin{minipage}[t]{0.5\linewidth}
\centering
\includegraphics[width=1\linewidth]{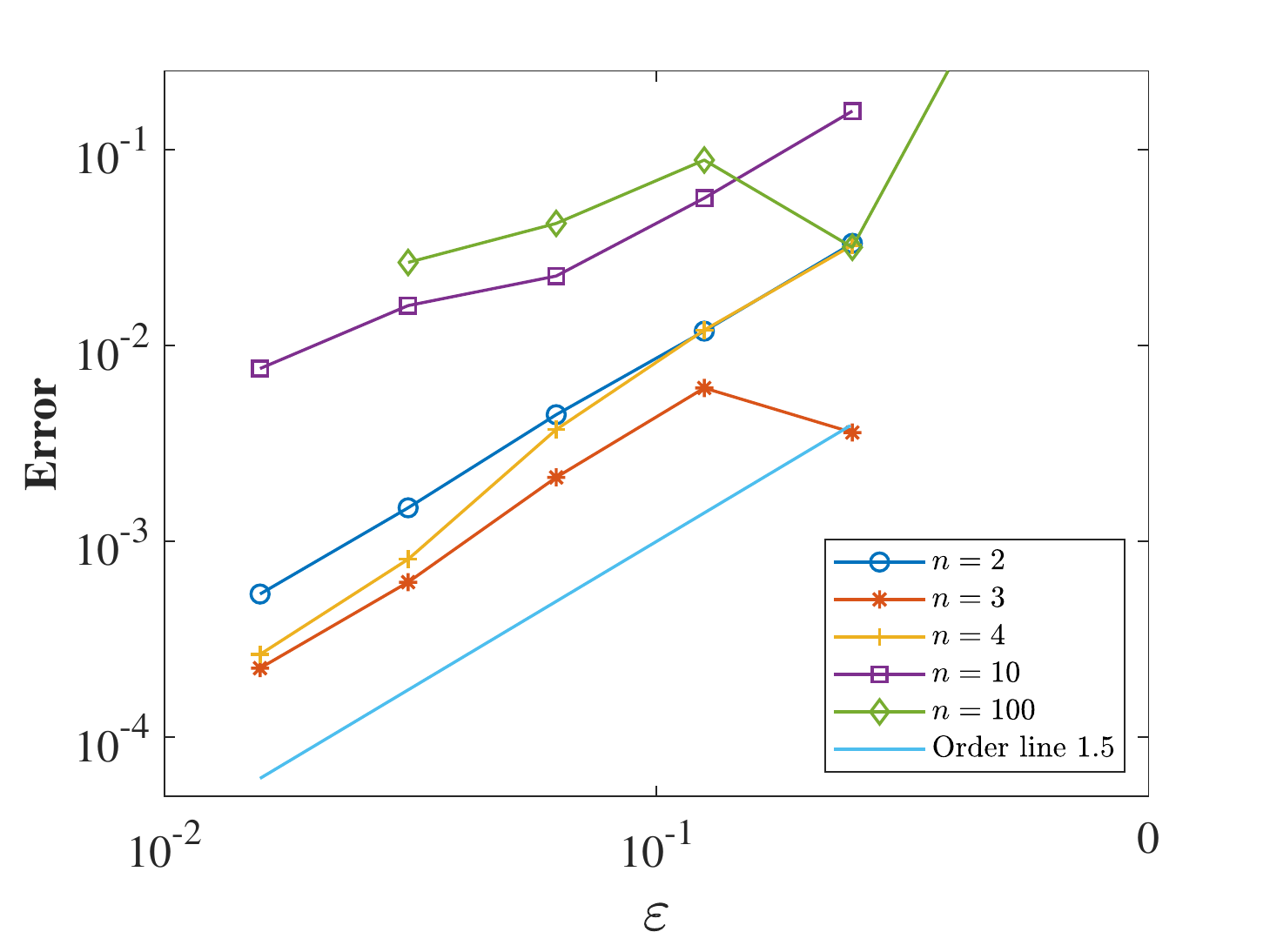}
\end{minipage}
}%
\centering
\captionsetup{font={footnotesize}}
\caption{Numerical errors of the scheme (\ref{eq:approSolution2}) for  Example \ref{example3}.
It shows that if $u(t,\texttt{{\rm{x}}})\in C^{1,\beta}([0,T]\times\mathbb{R}^{n}),\,\beta\in[2,3]$,
the convergent order is $2s\wedge (\lfloor\beta\rfloor-2s)$.} \label{2fig1}
\end{figure}
\section{Conclusion}
 We propose Monte Carlo method to solve the Dirichlet problem for the
 parabolic equation with fractional Laplacian (\ref{eq:terminalboundaryproblem}).
 First, we give the probabilistic representation of the solution to parabolic equation
 (\ref{eq:terminalboundaryproblem}) which is related to stochastic differential
 equations driven by symmetric stable L\'{e}vy process with jump.
 Then we obtain two jump-adapted schemes in which time discretization is
 based on jump times of L\'{e}vy process to approximate L\'{e}vy driven stochastic
 differential equations. The first scheme removes small jumps of the $2s$-symmetric
 stable process while the second scheme replaces small jumps by the simple process
 $\sigma_{\varepsilon}\sqrt{t}\xi$. Based on these two schemes, we give two numerical
 algorithms to solve parabolic equations (\ref{eq:terminalboundaryproblem}).
 Convergence theorems for both schemes are proved.
 The convergence order of the second scheme (\ref{eq:approSolution2}) is higher than the
 first one (\ref{eq:approSolution}), while the second scheme requires better
 regularity of the solution. Numerical experiments verify the theoretical analysis and
 show the efficiency of the proposed algorithms.

\section*{Acknowledgements} The authors wish to thank EiC Prof Karl Sabelfeld for his invaluable
suggestions and making some references available.

\end{document}